\documentclass[12pt,a4paper,leqno]{amsart}

\title[The wave front set of the Wigner distribution]{The wave front set of the Wigner distribution and instantaneous frequency}

\author[P. Boggiatto]{Paolo Boggiatto}

\address{Dipartimento di Matematica, Universit{\`a} di Torino, Via Carlo Alberto 10, 10123 Torino (TO), Italy.}

\email{paolo.boggiatto@unito.it}

\author[A. Oliaro]{Alessandro Oliaro}

\address{Dipartimento di Matematica, Universit{\`a} di Torino, Via Carlo Alberto 10, 10123 Torino (TO), Italy.}

\email{alessandro.oliaro@unito.it}

\author[P. Wahlberg]{Patrik Wahlberg}

\address{Dipartimento di Matematica, Universit{\`a} di Torino, Via Carlo Alberto 10, 10123 Torino (TO), Italy.}

\email{patrik.wahlberg@unito.it}


\usepackage{latexsym}
\usepackage{amsmath}
\usepackage{amssymb}
\usepackage{amsthm}
\usepackage{amsfonts}
\usepackage{mathrsfs}
\usepackage{calc}
\usepackage{cite}

\numberwithin{equation}{section}          

\newtheorem{thm}{Theorem}
\numberwithin{thm}{section}

\newcommand{\rubrik}{}
\newtheorem{prop}[thm]{Proposition}
\newtheorem{cor}[thm]{Corollary}
\newtheorem{lem}[thm]{Lemma}

\theoremstyle{definition}

\newtheorem{defn}[thm]{Definition}
\newtheorem{example}[thm]{Example}

\theoremstyle{remark}

\newtheorem{rem}[thm]{Remark}              

\newcommand{\pd}[1] {\partial ^#1}
\newcommand{\pdd}[2] {\partial_{#1} ^#2}

\newcommand{\ro}{\mathbb R}
\newcommand{\no}{\mathbb N}
\newcommand{\rr}[1]{\mathbb R^{#1}}
\newcommand{\nn}[1]{\mathbb N^{#1}}

\newcommand{\co}{\mathbb C}
\newcommand{\cc}[1]{\mathbb C^{#1}}

\newcommand{\ep}{\varepsilon}

\newcommand{\supp}{\operatorname{supp}}
\newcommand{\sgsupp}{\operatorname{sing\, supp}}

\newcommand{\eabs}[1]{\langle #1\rangle}

\newcommand{\wh}{\widehat}

\newcommand{\re}{\rm Re \ }
\newcommand{\im}{\rm Im \ }

\begin{document}

\begin{abstract}
We prove a formula expressing the gradient of the phase function of a function $f: \rr d \mapsto \co$ as a normalized first frequency moment of the Wigner distribution for fixed time. The formula holds when $f$ is the Fourier transform of a distribution of compact support, or when $f$ belongs to a Sobolev space $H^{d/2+1+\ep}(\rr d)$ where $\ep>0$.
The restriction of the Wigner distribution to fixed time is well defined provided a certain condition on its wave front set is satisfied. Therefore we first study the wave front set of the Wigner distribution of a tempered distribution.
\end{abstract}

\maketitle

\section{Introduction}

This paper treats a time-frequency version of the following trivial observation in Fourier analysis. Let $f(t) = C e^{2 \pi i \xi_0 \cdot t}$, $C \in \co \setminus 0$, $\xi_0 \in \rr d$, be a nonzero complex multiple of a character on $\rr d$. Its Fourier transform is $\wh f = C \delta_{\xi_0}$ so the frequency $\xi_0$ may be expressed using the Fourier transform as the normalized first order moment formula
\begin{equation}\label{instfreqchar0}
\xi_0 = \frac{\langle \wh f, \xi \rangle}{\langle \wh f, 1 \rangle}
\end{equation}
where $\langle \wh f, \xi \rangle$ is the vector $\langle \wh f, \xi \rangle = (\langle \wh f, \xi_j \rangle)_{j=1}^d \in \rr d$ and $\xi_j: \rr d \mapsto \ro$ is coordinate function $j$, $1 \leq j \leq d$.

We will deduce a time-frequency version of this formula for more general functions, which looks like
\begin{equation}\label{instfreqformula1}
\frac1{2 \pi} \nabla \arg f(t) = \frac{\langle W_f(t,\cdot), \xi \rangle}{\langle W_f(t,\cdot), 1 \rangle} \quad \forall t \in \rr d: \ f(t) \neq 0, \quad f \in \mathscr F \mathscr E'(\rr d).
\end{equation}
In the formula \eqref{instfreqformula1} $W_f$ denotes the Wigner distribution, defined by
\begin{equation}\nonumber
W_f (t,\xi) = \int_{\rr d} f (t+\tau/2) \overline{f(t-\tau/2)} e^{-2 \pi i \tau \cdot \xi} d\tau
\end{equation}
for $f \in \mathscr S(\rr d)$.
For functions in $\mathscr F \mathscr E'(\rr d)$ which are not multiples of characters $e^{2 \pi i \xi_0 \cdot t}$, the frequency is not well defined. Therefore it is replaced in \eqref{instfreqformula1} by the natural generalization
\begin{equation}\nonumber
\frac1{2 \pi} \nabla \arg f(t) = \frac1{2 \pi} \left( \partial_j \arg f(t) \right)_{j=1}^d,
\end{equation}
that is, a normalized gradient of the phase function. We use the term instantaneous frequency, taken from the engineering literature \cite{Cohen1,Flandrin1}, for this quantity. Thus \eqref{instfreqformula1} may be seen as a time-frequency version of the observation \eqref{instfreqchar0}. We shall also prove a version of \eqref{instfreqformula1} for functions $f$ that belong to a Sobolev space $H^{d/2+1+\ep}(\rr d)$ where $\ep>0$. Then the distribution actions $\langle W_f(t,\cdot), \xi \rangle$ and $\langle W_f(t,\cdot), 1 \rangle$ may be replaced by Lebesgue integrals.

In order to prove \eqref{instfreqformula1} we need to restrict the Wigner distribution as $W_f \mapsto W_f(t,\cdot)$ to fixed time $t \in \rr d$. For $f \in \mathscr S'(\rr d)$, the restriction is a map $\mathscr S'(\rr {2d}) \mapsto \mathscr D'(\rr d)$. Restriction of a distribution to a submanifold is possible under certain conditions on the wave front set \cite{Hormander1}. More precisely, the restriction defines a well defined distribution provided the normal bundle of the submanifold has empty intersection with the wave front set of the distribution. Thus we are led to study the wave front set of the Wigner distribution first. We pursue this study in somewhat greater generality than actually needed in order to prove formula \eqref{instfreqformula1}.

We define the space $WFW^\perp$ of tempered distributions such that the wave front set of the Wigner distribution is directed purely in the frequency direction, and the space $WFW^{\neq}$ of tempered distributions such that the wave front set of the Wigner distribution is nowhere parallel to the time direction. The latter space admits restriction $W_f \mapsto W_f(t,\cdot)$. We show the inclusions $C_{\rm slow}^\infty \subseteq WFW^\perp$ and $C_{\rm slow}^\infty \subseteq V_{\rm con} \subseteq WFW^{\neq}$. Here $C_{\rm slow}^\infty$ denotes the space of smooth functions on $\rr d$ such that a derivative of any order is bounded by a constant times a fixed polynomial. $C_{\rm slow}^\infty$ contains $\mathscr{F E'}$. The space $V_{\rm con}$ is the subspace of $\mathscr S'(\rr d)$ such that the short-time Fourier transform (STFT) $V_\varphi f$ satisfies
$$
\sup_{|\eta|>B|x|} \eabs{(x,\eta)}^n |V_\varphi f(x,\eta)| < \infty \quad \forall n>0,
$$
for some $B>0$. This means that the STFT decays polynomially in a conic neighborhood of the frequency axis.

Recently Guo, Molahajloo and Wong have studied the instantaneous frequency and its relation to the modified Stockwell transform \cite{Guo1}. For other aspects of wave front sets and time-frequency analysis we refer to the recent papers \cite{Coriasco1,Coriasco2,Pilipovic1,Pilipovic2}.

\section{Preliminaries}

The Schwartz space $\mathscr S(\rr d)$ consists of smooth functions such that a derivative of any order multiplied by any polynomial is uniformly bounded. Its topological dual $\mathscr S'(\rr d)$ is the space of tempered distributions. We denote by $C_c^\infty (\rr d)$ the space of smooth and compactly supported functions, and $\mathscr D' (\rr d)$ is its topological dual, the space of distributions. The compactly supported distributions are denoted $\mathscr E'(\rr d)$. 

The normalization of the Fourier transform for functions $f\in \mathscr S(\rr d)$ used in this paper is
$$
\mathscr{F} f(\xi) = \wh f(\xi) = \int_{\rr d} f(x) e^{-2 \pi i x \cdot \xi} dx,
$$
where $x \cdot \xi$ denotes the inner product on $\rr d$. The inverse Fourier transform is then
$$
\mathscr{F}^{-1} f(x) = \int_{\rr d} f(\xi) e^{2 \pi i x \cdot \xi} d\xi.
$$
The Fourier transform extends by duality to a homeomorphism on $\mathscr S'(\rr d)$. For $s \in \ro$ the Sobolev space $H^s(\rr d)$ is defined as the subspace of $f \in \mathscr S'(\rr d)$ such that $\wh f \in L_{\rm loc}^2(\rr d)$ and
$$
\| f \|_{H^s} = \left( \int_{\rr d} \eabs{\xi}^{2s} | \wh f(\xi)|^2 d \xi  \right)^{1/2} < \infty
$$
where $\eabs{\xi}=(1+|\xi|^2)^{1/2}$.

Translation is denoted by $(T_y f)(x)=f(x-y)$ and modulation by $M_\xi f(x) = e^{2 \pi i x \cdot \xi} f(x)$ for functions of one $\rr d$ variable, and by $(T_{y,w} f)(x,z)=f(x-y,z-w)$, $M_{\xi,\eta} f(x,y) = e^{2 \pi i (x \cdot \xi + y \cdot \eta)} f(x,y)$, respectively, for functions of two $\rr d$ variables. The short-time Fourier transform (STFT) of $f \in \mathscr S' (\rr d)$ with respect to a window function $\varphi \in \mathscr S(\rr d)$ \cite{Grochenig1,Folland1} is defined by
\begin{equation}\nonumber
V_\varphi f(t,\xi) = ( f, M_\xi T_t \varphi )
\end{equation}
where $(\cdot,\cdot)$ denotes the conjugate (for consistency with the $L^2$-product) linear action of $\mathscr D' (\rr d)$ on $C_c^\infty (\rr d)$, or $\mathscr S'(\rr d)$ acting on $\mathscr S(\rr d)$.
Thus the inner product on $L^2$, also denoted by $(\cdot,\cdot)$, is conjugate linear in the second variable. We denote the \emph{linear} (without conjugation) action of distributions on test functions by $\langle \cdot, \cdot \rangle$, and thus $(u,\varphi)=\langle u,\overline{\varphi} \rangle$.

We use the symbol $C$ for a positive constant that may change value over inequalities and equalities.
The space $C_0(\rr d)$ consists of continuous functions that vanish at infinity. Thus $f \in C_0(\rr d)$ means that for any $\ep>0$ there exists a compact set $K_\ep \subset \rr d$ such that $x \notin K_\ep \Rightarrow |f(x)|<\ep$.
The symbol $C^k(\rr d)$ is the space of functions such that all partial derivatives of order not greater than $k$ are continuous everywhere, and $C^\infty(\rr d) = \bigcap_{k \geq 1} C^k (\rr d)$.

We recall the definition of the $C^\infty$ wave front set of $u \in \mathscr D'(\rr d)$ \cite{Folland1,Hormander1}, denoted  $WF(u)$. Let $\Gamma \subseteq \rr d \setminus 0$ denote an open conic subset, where conic means $\xi \in \Gamma \Rightarrow a \xi \in \Gamma$ for all $a>0$. The wave front set is defined as the complement
\begin{equation}\nonumber
\begin{aligned}
WF(u) & = \left( \rr d \times \rr d \setminus 0 \right) \setminus \\
& \Big\{ (x,\xi) : \ \exists \psi \in C_c^\infty(\rr d): \psi(x) \neq 0 ,
\exists \Gamma \subset \rr d \setminus 0: \ \xi \in \Gamma,  \\
& \sup_{\eta \in \Gamma} \eabs{\eta}^n |\wh{\psi u}(\eta)| < \infty \quad \forall n>0 \Big\}.
\end{aligned}
\end{equation}
We have $P_1 WF(u) = \sgsupp (u)$ where $P_1$ denotes the projection on the first $\rr d$ variable. The singular support $\sgsupp (u)$ is the complement of the largest open set where $u$ is $C^\infty$.

The cross-Wigner distribution is defined by
\begin{equation}\label{wd1}
\begin{aligned}
W_{f,g} (t,\xi) & = \int_{\rr d} f (t+\tau/2) \overline{g(t-\tau/2)} e^{-2 \pi i \tau \cdot \xi} d\tau \\
& = \left( \mathscr F _2 (f \otimes \overline{g} \circ \kappa) \right) (t,\xi), \quad f,g \in \mathscr S(\rr d), \quad t,\xi \in \rr d,
\end{aligned}
\end{equation}
where
$$
\kappa(x,y)=(x+y/2,x-y/2)
$$
and $\mathscr F_2$ denotes partial Fourier transformation in the second $\rr d$ variable. The Wigner distribution of a single function is denoted by $W_f=W_{f,f}$. The original definition \eqref{wd1} for $f,g \in \mathscr S(\rr d)$ extends to $f,g \in \mathscr S'(\rr d)$, and then $W_{f,g} \in \mathscr S'(\rr {2d})$.

The Wigner distribution has been studied thoroughly from many points of view, in particular quantum mechanics \cite{Cohen1,Folland1}, pseudo-differential calculus \cite{Folland1} and signal analysis \cite{Grochenig1}.
In signal analysis the Wigner distribution has been studied as a way to represent signals simultaneously in the time and frequency variables. The Wigner distribution satisfies many properties requested by an ideal time-frequency representation, among which the most important include the following, which holds for example when $f \in \mathscr S(\rr d)$.
\begin{align}
& W_{M_\eta T_x f} = T_{x,\eta} W_f, \quad x, \eta \in \rr d, \label{covariance1} \\
& \iint_{\rr {2d}} W_f(t,\xi) d t d \xi = \| f \|_{L^2(\rr d)}^2, \label{energy1} \\
& \int_{\rr d} W_f(t,\xi) d \xi = |f(t)|^2, \label{marginal1} \\
& \int_{\rr d} W_f(t,\xi) d t = |\wh f(\xi)|^2, \label{marginal2} \\
& \frac1{2 \pi} \nabla \arg f(t) = \frac{\int_{\rr d} \xi W_f(t,\xi) d\xi}{\int_{\rr d} W_f(t,\xi) d\xi} \quad \forall t \in \rr d: \ f(t) \neq 0. \label{instfreqwigner1}
\end{align}
Formula \eqref{covariance1} says that the Wigner distribution respects time-frequency shifts, \eqref{energy1} says that its integral equals the squared energy of the function, and \eqref{marginal1}, \eqref{marginal2} say that $W_f$ has the correct marginal properties. This admits the interpretation of $W_f$ as a distribution of the energy of $f$ over $(t,\xi) \in \rr {2d}$. Alternatively, in quantum mechanics, $W_f$ is interpreted as a phase space probability density. However, these interpretations are in general not possible, since $W_f$ is not nonnegative everywhere unless $f$ is a generalized Gaussian of the form
\begin{equation}\label{hudsongauss1}
f(t) = \exp(- \pi t \cdot A t + 2 \pi b \cdot t + c)
\end{equation}
where $c \in \co$, $b \in \cc d$, $A \in \cc {d \times d}$ is invertible and $\re A > 0$. This is Hudson's theorem \cite{Folland1,Grochenig1}. The function $f$ has Wigner distribution
\begin{equation}\nonumber
\begin{aligned}
W_f(t,\xi) & = C \exp\left( - 2 \pi \left( t \cdot \re A t - 2 \re b t \right) \right) \\
& \times \exp\left( - 2 \pi (\xi + \im A t - \im b) \cdot (\re A)^{-1} (\xi + \im A t - \im b) ) \right)
\end{aligned}
\end{equation}
and
\begin{equation}\nonumber
\frac1{2 \pi} \nabla \arg f(t) = -\im A t + \im b.
\end{equation}
Thus $W_f$ is a function mainly concentrated along the submanifold
$$
\{ (t, (2 \pi)^{-1} \nabla \arg f(t)): \ t \in \rr d \} \subseteq \rr {2d}.
$$
The formula \eqref{instfreqwigner1} expresses a generalization of this observation, in the sense that the instantaneous frequency is a normalized first order frequency moment of the Wigner distribution, for fixed $t \in \rr d$.
It is well known in the applied literature \cite{Cohen1} where it is derived without precise assumptions for $d=1$.

Janssen \cite{Janssen1,Janssen2} has studied the question whether the Wigner distribution (for $d=1$) may be concentrated on a curve in the phase space. Under some assumptions on the curve it turns out that it must be a straight line and the distribution $f$ is either a multiple of a Dirac distribution or a degenerate Gaussian of the form \eqref{hudsongauss1} with $\re A=0$. This means that $W_f$
is of the form $W_f(t,\xi) = C \delta_0(\xi - (2 \pi)^{-1} \nabla \arg f(t))$ (assuming $b=0$), i.e. supported on the subspace $\{(t,(2 \pi)^{-1} \nabla \arg f(t)): \ t \in \ro \} \subseteq \rr {2}$. For other functions the Wigner distribution gives a dispersion. However, \eqref{instfreqwigner1} shows that the mean value of $W_f(t,\cdot)$ for fixed $t \in \rr d$ agrees with intuition since it equals the instantaneous frequency.

This paper concerns assumptions that imply that \eqref{instfreqwigner1} holds true. It is relatively straightforward to relax $f \in \mathscr S(\rr d)$ into $f \in H^{d/2+1+\ep}(\rr d)$ where $\ep>0$ (see Proposition \ref{ifwignerprop1}), since $W_f$ is then a continuous function which admits restriction $W_f \mapsto W_f(t,\cdot)$ without problem. As another assumption we use $f \in \mathscr{F E'}(\rr d)$ (see Proposition \ref{ifwignerprop2}), which implies that $W_f \in \mathscr S'(\rr {2d})$. This case is more subtle since restriction of a distribution to a submanifold is not always possible. However, there are conditions on the wave front set that are sufficient for a restriction to be well defined and continuous. Hence we need to investigate the wave front set of the Wigner distribution. This problem is a generalization of the study of the singular support of $W_f$, considered by Janssen \cite{Janssen1,Janssen2}.

The paper is organized as follows. We investigate the wave front set of the Wigner distribution in Section \ref{wfw}. In Section \ref{restriction} we study sufficient conditions for the restriction operator $W_f \mapsto W_f(t,\cdot)$ to be well defined for $t \in \rr d$ fixed. Finally in Section \ref{wignermoment} we prove formula \eqref{instfreqformula1} for $f \in \mathscr {F E'}(\rr d)$ and for $f \in H^{d/2+1+\ep}(\rr d)$ where $\ep>0$.

\section{The wave front set of the Wigner distribution}\label{wfw}

Let $f \in \mathscr S'(\rr d)$. Two natural questions are to compare $\sgsupp (f)$ and $P_1 \sgsupp (W_f)$, and to compare
$WF(f)$ and $P_{1,3} WF(W_f)$. Here $P_1: \rr {2d} \mapsto \rr {d}$ denotes the projection on the first variable, $P_{1}(t,\xi)=t$, $t,\xi \in \rr d$, and $P_{1,3}: \rr {4d} \mapsto \rr {2d}$ denotes the projection on the first and third variables, $P_{1,3}(t,\xi; \eta,x)=(t,\eta)$, $t,\xi,\eta,x \in \rr d$.

Let $f \equiv 1$. Then $\sgsupp (f) = \emptyset$ and $WF(f) = \emptyset$. Moreover, $W_f = 1 \otimes \delta_0$ so $\sgsupp (W_f) = \rr d \times 0$ and $WF(W_f) = \{ (t,0; 0,x): \ t \in \rr d,\ x \in \rr d \setminus 0\}$ (see Example \ref{example1}). Hence $P_1 \sgsupp (W_f) = \rr d$, $P_{1,3} WF(W_f) = \rr d \times 0$. Thus
\begin{equation}\nonumber
\begin{aligned}
P_1 \sgsupp (W_f) & \nsubseteq \sgsupp (f), \\
P_{1,3} WF(W_f) & \nsubseteq WF(f), \quad f \in \mathscr S'(\rr d).
\end{aligned}
\end{equation}

We do not know whether any of the following inclusions hold for $f \in \mathscr S'(\rr d)$.
\begin{align}
\sgsupp (f) & \subseteq P_1 \sgsupp (W_f), \label{openproblem1} \\
WF(f) & \subseteq P_{1,3} WF(W_f). \label{openproblem2}
\end{align}
Note that the inclusion \eqref{openproblem2} is stronger than \eqref{openproblem1}. In fact, assume \eqref{openproblem2}. We have $\sgsupp (W_f) = P_{1,2} WF(W_f)$ and therefore $P_1 WF(W_f) = P_1 \sgsupp (W_f)$.
The assumption implies
\begin{equation}\nonumber
\begin{aligned}
\sgsupp (f) & = P_1 WF(f) \subseteq P_1 P_{1,3} WF(W_f) = P_1 WF(W_f) \\
& = P_1 \sgsupp (W_f),
\end{aligned}
\end{equation}
and thus \eqref{openproblem1} follows from \eqref{openproblem2}.

To produce a counterexample to the inclusion \eqref{openproblem1}, it suffices to find a function $f$ which is not $C^\infty$ at $t=0$ whose Wigner distribution $W_f$ is $C^\infty$ in a neighborhood of $(0,\xi)$ for all $\xi \in \rr d$. There is some weak evidence that such a function may exist: in fact the Wigner distribution is regularizing in the sense that
\begin{equation}\label{wignercont1}
W: L^2(\rr d) \times L^2(\rr d) \mapsto C_0(\rr {2d})
\end{equation}
continuously. To see this, let $(f_n)$ be a sequence of functions $f_n \in C_c^\infty(\rr d)$ such that $\| f-f_n \|_{L^2} \rightarrow 0$ as $n \rightarrow \infty$. Since $W_{f_n} \in \mathscr S(\rr {2d})$ and since $\| W_{f,g} \|_{L^\infty} \leq 2^d \| f \|_{L^2} \| g \|_{L^2}$ by the Cauchy--Schwarz inequality, we see that $W_f$ is the uniform limit of $C_0(\rr {2d})$ functions, and since $C_0(\rr {2d})$ equipped with the supremum norm is a Banach space we have $W_f \in C_0(\rr {2d})$.

A characterization of the wave front set $WF(f)$ in terms of the asymptotic behavior of $W_f$ for large frequencies is given in \cite[Corollary 3.28]{Folland1}. Let $\phi \in \mathscr S (\rr d)$ be even, nonzero and define the dilation $\phi^s(x) = s^{d/4} \phi(s^{1/2} x)$ for $s>0$. Then $(t_0,\xi_0) \notin WF(f)$ if and only if there exists a neighborhood $U$ of $(t_0,\xi_0)$, conic in the second variable, such that for any $a \geq 1$
\begin{equation}\label{wavefrontwigner1}
\sup_{(t,\xi) \in U, \ a^{-1} \leq |\xi| \leq a, \ s \geq 1} s^{n} |(W_f * W_{\phi^s}) (t,s\xi) | < \infty \quad \forall n \geq 1.
\end{equation}
This criterion says that the convolution $W_f * W_{\phi^s}$ decreases faster than any polynomial in the frequency direction in a neighborhood of $(t_0,\xi_0)$, conic in the second variable. Note that the function $W_{\phi^s}(t,\xi) = W_\phi (s^{1/2} t,s^{-1/2} \xi)$, with which $W_f$ is convolved, concentrates around zero in the $t$ variable and spreads out in the $\xi$ variable as $s \rightarrow +\infty$.

In particular we have that $f$ is smooth in a neighborhood of $t_0$ if and only if there exists a neighborhood $V$ of $t_0$ such that for any $a \geq 1$
\begin{equation}\nonumber
\sup_{t \in V, \ a^{-1} \leq |\xi| \leq a, \ s \geq 1} s^{n} |(W_f * W_{\phi^s}) (t,s\xi) | < \infty \quad \forall n \geq 1,
\end{equation}
which means that $W_f * W_{\phi^s}$ decreases faster than any polynomial in any frequency direction in a neighborhood of $t_0$.

The criterion \eqref{wavefrontwigner1} says that the wave front set $WF(f)$ can be characterized by $W_f$. Roughly speaking, the microregularity of $f$ is characterized by the asymptotic behavior of $W_f$ at infinity in conic frequency domains. (More precisely, this behavior concerns the convolution $W_f * W_{\phi^s}$ and not $W_f$.)

In the remainder of the this section we will investigate the wave front set of the Wigner distribution $W_f$ for $f \in \mathscr S'(\rr d)$. We introduce the following two subspaces of $\mathscr S'(\rr d)$ for this purpose. Here we understand by subspace a subset which is not necessarily linear\footnote{We do not know whether $WFW^\perp$ or $WFW^{\neq}$ are linear subspaces of $\mathscr S'$.}.

\begin{defn}
\begin{equation}\nonumber
WFW^\perp (\rr d) = \{ f \in \mathscr S'(\rr d): \ WF(W_f) \subseteq \rr {2d} \times (0 \times \rr d \setminus 0) \}.
\end{equation}
\end{defn}
Thus $WFW^\perp (\rr d)$ consists of the tempered distributions such that the wave front set of the Wigner distribution is directed purely in the frequency direction (or is empty).

\begin{defn}
\begin{equation}\nonumber
WFW^{\neq} (\rr d) = \{ f \in \mathscr S'(\rr d): \ WF(W_f) \cap \rr {2d} \times (\rr d \setminus 0 \times 0 )= \emptyset \}.
\end{equation}
\end{defn}
The definition says that $WFW^{\neq}(\rr d)$ consists of the tempered distributions such that the wave front set of the Wigner distribution does not contain vectors purely in the time direction. Obviously
\begin{equation}\nonumber
WFW^\perp (\rr d) \subseteq WFW^{\neq} (\rr d).
\end{equation}

The following example shows that for $f(t)=\exp(\pi i t \cdot A t)$, $A$ symmetric, we have
$$
f \in WFW^{\neq} (\rr d) \setminus WFW^\perp (\rr d), \quad A \in \rr {d \times d} \setminus 0.
$$

\begin{example}\label{example1}
Let $f(t)=\exp(\pi i t \cdot A t)$ where $A \in \rr {d \times d}$ is a symmetric matrix ($f$ is sometimes called a \emph{chirp} \cite{Grochenig1}). Then
\begin{equation}\nonumber
W_f(t,\xi) = \int_{\rr d} \exp(-2 \pi i \tau \cdot( \xi - A t) ) d \tau = \delta_0(\xi-A t),
\end{equation}
\begin{equation}\nonumber
\sgsupp (W_f) = \{ (t, A t), \ t \in \rr d \}.
\end{equation}
Since the transformation $T: \rr {2d} \mapsto \rr d$, $T(t,\xi)=\xi-A t$ has surjective differential, the distribution $W_f=\delta_0 \circ T \in \mathscr D'(\rr {2d})$ is well defined \cite[Theorem 6.1.2]{Hormander1}.
Let $\varphi \in C_c^\infty(\rr {2d})$ and let $(\psi_n) \subset C_c^\infty(\rr d)$ be a sequence converging in $\mathscr D'(\rr d)$ to $\delta_0$. The continuity statement of \cite[Theorem 6.1.2]{Hormander1} gives
\begin{equation}\nonumber
\begin{aligned}
\wh {\varphi W_f}(\eta,x) & = \langle \delta_0 \circ T, M_{-\eta,-x} \varphi \rangle
= \lim_{n \rightarrow \infty} \langle \psi_n \circ T, M_{-\eta,-x} \varphi \rangle \\
& = \lim_{n \rightarrow \infty} \iint_{\rr {2d}} \psi_n(\xi-A t) \varphi(t,\xi) e^{- 2 \pi i (\eta \cdot t + x \cdot \xi)} dt d\xi \\
& = \lim_{n \rightarrow \infty} \iint_{\rr {2d}} \psi_n(\xi) \varphi(t,\xi+A t) e^{- 2 \pi i \left( t \cdot (\eta + A x) + x\cdot \xi \right) } dt d\xi \\
& = \int_{\rr d} \varphi(t,A t) e^{- 2 \pi i t \cdot (\eta + A x) } dt \\
& = \wh \chi (\eta + A x)
\end{aligned}
\end{equation}
where $\chi(t) = \varphi(t,A t) \in C_c^\infty(\rr d)$.

Let $(t,\xi) \in \sgsupp (W_f)$, i.e. $\xi=At$ and suppose $\varphi(t,A t) \neq 0$.
If $\eta_0 + A x_0 \neq 0$ then there is a conic neighborhood $\Gamma$ containing $(\eta_0,x_0)$ such that $| \eta + A x | \geq C>0$ when $|(\eta,x)|=1$ and $(\eta,x) \in \Gamma$. This gives $| \eta + A x | \geq C | (\eta,x)|$ for $(\eta,x) \in \Gamma$ and thus, using the fact that $\wh \chi \in \mathscr S(\rr d)$,
\begin{equation}\nonumber
\sup_{(\eta,x) \in \Gamma} |(\eta,x)|^n |\wh {\varphi W_f}(\eta,x)| \leq C_n \sup_{(\eta,x) \in \Gamma} |(\eta,x)|^n | \eta + A x |^{-n} < \infty
\end{equation}
for any $n>0$, so $(t,A t; \eta_0,x_0) \notin WF(W_f)$.

On the other hand, we may use the following result \cite{Hormander1} for $u \in \mathscr D'(\rr d)$ and $\varphi_j \in C_c^\infty(\rr d)$. If $\varphi_j(x) \neq 0$ for all $j \geq 1$, $\supp(\varphi_j) \rightarrow \{ x\}$, and $\wh{ \varphi_j u}$ does not decay polynomially in any conical neighborhood of $\xi$ for any $j$, then $(x,\xi) \in WF(u)$.

Let $\varphi$ satisfy $\varphi(t,A t) \neq 0$ and $\varphi \geq 0$. Then $\wh \chi(0) = \int \varphi(u,A u) du > 0$ and $\wh{\varphi W_f}(-Ax,x) = \wh \chi(0) \neq 0$ for any $x \in \rr d$. Therefore $\wh{\varphi W_f}$ is not polynomially decreasing in any conic neighborhood of the manifold $\{ (\eta,x): \ \eta+Ax = 0 \}$. By shrinking the support of $\varphi$ we may thus conclude
\begin{equation}\nonumber
WF(W_f) = \{ (t, A t; -A x,x), \ t \in \ro,\ x \in \ro \setminus 0 \}.
\end{equation}
For $A=0$ we thus have $f \in WFW^\perp (\rr d)$ but if $A \neq 0$ then $f \in WFW^{\neq} (\rr d) \setminus WFW^\perp (\rr d)$.
\end{example}

We are interested in $WFW^{\neq} (\rr d)$ since $f \in WFW^{\neq} (\rr d)$ is a sufficient condition for the restriction operator $W_f \mapsto W_f(t,\cdot)$ to be well defined from $\mathscr D'(\rr {2d})$ to $\mathscr D'(\rr d)$ for all $t \in \rr d$ (see Section \ref{restriction}). We are interested in $WFW^\perp (\rr d)$ since it is easier to prove inclusions of familiar spaces in $WFW^\perp (\rr d)$.

Next we define a linear space of smooth functions, whose derivatives are slowly increasing, uniformly with respect to the order of derivation.

\begin{defn}
The space $C_{\rm slow}^\infty(\rr d)$ consists of smooth functions such that a derivative of any order $\alpha \in \nn d$ is bounded by a constant $C_\alpha$ times a fixed polynomial, that is $f \in C_{\rm slow}^\infty(\rr d)$ if there exists $N>0$ such that
\begin{equation}\label{conditionf}
\vert \partial^\alpha f(x)\vert\leq C_\alpha \langle x\rangle^N \quad \forall x \in \rr d, \  C_\alpha>0, \ \alpha \in \nn d.
\end{equation}
\end{defn}

Note that $C_{\rm slow}^\infty(\rr d) \subseteq \mathcal O_M^d(\rr d)$, where $\mathcal O_M^d(\rr d)$ is the space of smooth functions such that the derivatives satisfy \eqref{conditionf} where $N$ may depend on $\alpha$. For more information on $\mathcal O_M^d(\rr d)$ we refer to \cite{Reed1}.

\begin{lem}\label{fouriercompact1}
\begin{equation}\nonumber
\mathscr F \mathscr E'(\rr d) \subseteq C_{\rm slow}^\infty(\rr d).
\end{equation}
\end{lem}
\begin{proof}
Let $f \in \mathscr F \mathscr E'(\rr d)$ which means that there exists $u \in \mathscr E'(\rr d)$ such that $f(x) = (u,\phi_x)$ where $\phi_x (\xi)=\exp(2 \pi i x \cdot \xi)$ \cite[Theorem 7.1.14]{Hormander1}. Denote by $N$ the finite order of the distribution $u \in \mathscr E'(\rr d)$. By \cite[Theorem 2.1.3]{Hormander1} we have $\pd \alpha  f(x) = (u,\pdd x \alpha \phi_x)$ for any $\alpha \in \nn d$. For some compact set $K \subset \rr d$ containing the support of $u$, this yields
\begin{equation}\nonumber
\begin{aligned}
|\pd \alpha  f(x)| = |(u,\pdd x \alpha \phi_x)|
& \leq C_\alpha \sum_{|\beta| \leq N}  \sup_{\xi \in K} | \pdd x \alpha \pdd \xi \beta \phi_x (\xi) | \\
& \leq C_\alpha \eabs{x}^N, \quad \alpha \in \nn d.
\end{aligned}
\end{equation}
\end{proof}

\begin{rem}
If $f \equiv 1 = \mathscr F \delta_0$ then $W_f = 1 \otimes \delta_0$, and $WF(W_f) = (\rr d \times 0) \times (0 \times \rr d \setminus 0)$ according to Example \ref{example1}. This shows that the wave front set of the Wigner distribution of a function $f \in \mathscr F \mathscr E'(\rr d)$ in general is nonempty.
\end{rem}

\begin{rem}
Note that the inclusion $\mathscr F \mathscr E'(\rr d) \subseteq C_{\rm slow}^\infty(\rr d)$ is strict. For instance $C_c^\infty(\rr d) \subseteq C_{\rm slow}^\infty(\rr d) \setminus \mathscr F \mathscr E'(\rr d)$.
\end{rem}

Next we shall make some preparations in order to introduce a linear space which is larger than $C_{\rm slow}^\infty(\rr d)$.

\begin{lem}\label{wignerfourier1}
If $f \in \mathscr S'(\rr d)$ and $\varphi \in \mathscr S(\rr d)$ then
\begin{equation}\nonumber
\wh {W_f W_\varphi} (\eta,x) = V_\varphi f(-x/2,\eta/2) \overline{ V_\varphi f(x/2,-\eta/2) }, \quad \eta,x \in \rr d.
\end{equation}
\end{lem}
\begin{proof}
We compute
\begin{equation}\nonumber
\begin{aligned}
(\mathscr F_2^{-1} ( M_{\eta,x} W_\varphi )) \circ \kappa^{-1}(y,z) & = \int_{\rr d} e^{2 \pi i ( \xi \cdot (y-z) + \eta \cdot(y+z)/2 + x \cdot \xi)} W_\varphi \left( \frac{y+z}{2},\xi \right) d \xi \\
& = e^{\pi i \eta \cdot(y+z)} \mathscr F_2^{-1} W_\varphi \left( \frac{y+z}{2},y-z+x \right) \\
& = e^{\pi i \eta \cdot(y+z)} (\varphi \otimes \overline{\varphi} \circ \kappa) \left( \frac{y+z}{2},y-z+x \right) \\
& = e^{\pi i \eta \cdot(y+z)} \varphi (y+x/2) \overline{\varphi(z-x/2)} \\
& = M_{\eta/2} T_{-x/2} \varphi(y) \overline{M_{-\eta/2} T_{x/2} \varphi(z)}.
\end{aligned}
\end{equation}
This gives, using \cite[Lemma 7.4.1]{Hormander1} and the fact that $W_\varphi$ is real-valued \cite{Grochenig1},
\begin{equation}\nonumber
\begin{aligned}
\wh {W_f W_\varphi} (\eta,x) & = ( W_f, M_{\eta,x} W_\varphi ) \\
& = ( f \otimes \overline{f} \circ \kappa, \mathscr F_2^{-1}(  M_{\eta,x} W_\varphi ) ) \\
& = ( f \otimes \overline{f}, \mathscr F_2^{-1}(  M_{\eta,x} W_\varphi ) \circ \kappa^{-1} ) \\
& = ( f \otimes \overline{f}, M_{\eta/2} T_{-x/2} \varphi \otimes \overline{M_{-\eta/2} T_{x/2} \varphi} ) \\
& = V_\varphi f(-x/2,\eta/2) \overline{ V_\varphi f(x/2,-\eta/2) }.
\end{aligned}
\end{equation}
\end{proof}

\begin{lem}\label{convolutiondecay1}
Let $g \in C^\infty(\rr {2d})$ and suppose
\begin{equation}\nonumber
| g(t,\xi)| \leq C \eabs{(t,\xi)}^N \quad \mbox{for some $N>0$ and all $t,\xi \in \rr d$},
\end{equation}
$$
\Gamma = \{(t,\xi): |\xi|>C|t| \} \quad \mbox{for some $C>0$, and}
$$
\begin{equation}\nonumber
\sup_{(t,\xi) \in \Gamma} \eabs{(t,\xi)}^n | g(t,\xi) | < \infty \quad \forall n>0.
\end{equation}
If $\varphi \in \mathscr S(\rr {2d})$ then we have for any $C'>C$ and $\Gamma' = \{(t,\xi): |\xi|>C'|t| \}$
\begin{equation}\label{decaycone1}
\sup_{(t,\xi) \in \Gamma'} \eabs{(t,\xi)}^n | g * \varphi (t,\xi) | < \infty \quad \forall n>0.
\end{equation}
\end{lem}
\begin{proof}
We have
\begin{equation}\nonumber
\begin{aligned}
|g * \varphi (t,\xi)| \leq & \iint_{\eabs{(x,\eta)} \leq \eabs{(t,\xi)}^{1/2}} |g(t-x,\xi-\eta)| | \varphi (x,\eta)| dx d\eta \\
& \quad + \iint_{\eabs{(x,\eta)} > \eabs{(t,\xi)}^{1/2}} |g(t-x,\xi-\eta)| |\varphi (x,\eta)| dx d\eta \\
& := I_1 + I_2.
\end{aligned}
\end{equation}
Consider the first integral $I_1$.
For any $C'>C$ and the corresponding cone $\Gamma' = \{(t,\xi): \ |\xi|>C'|t| \} \subset \Gamma$ we have: $(t,\xi) \in \Gamma'$ and $\eabs{(x,\eta)} \leq \eabs{(t,\xi)}^{1/2}$ imply that $(t-x,\xi-\eta) \in \Gamma$ provided that $\eabs{(t,\xi)}$ is sufficiently large.
In fact, these assumptions imply
$$
\frac{|\xi - \eta|}{|t-x|} \geq \frac{|\xi| - \eabs{(t,\xi)}^{1/2}}{C'^{-1} |\xi| + \eabs{(t,\xi)}^{1/2}} = C' \frac{|\xi| - \eabs{(t,\xi)}^{1/2}}{ |\xi| + C' \eabs{(t,\xi)}^{1/2} }
$$
and the quotient approaches one as $\eabs{(t,\xi)} \rightarrow \infty$, because $|\xi|>C'|t|$. Thus we have for any integer $n \geq 0$
\begin{equation}\label{intuppsk1a}
\begin{aligned}
I_1 & \leq C_n \iint_{\eabs{(x,\eta)} \leq \eabs{(t,\xi)}^{1/2}} \eabs{(t-x,\xi-\eta)}^{-n} |\varphi (x,\eta)| dx d\eta \\
& \leq C_n \eabs{(t,\xi)}^{-n} \iint_{\rr {2d}} \eabs{(x,\eta)}^{n} |\varphi (x,\eta)| dx d\eta \\
& \leq C_n  \eabs{(t,\xi)}^{-n}, \quad (t,\xi) \in \Gamma', \quad C_n>0,
\end{aligned}
\end{equation}
provided $\eabs{(t,\xi)}$ is sufficiently large.

Next let us look at the second integral $I_2$. We have for any $L>0$
\begin{equation}\label{intuppsk2a}
\begin{aligned}
I_2 & \leq C_L \iint_{\eabs{(x,\eta)} > \eabs{(t,\xi)}^{1/2}} \eabs{(t-x,\xi-\eta)}^{N} \eabs{(x,\eta)}^{-L} dx d\eta \\
& \leq C_L \eabs{(t,\xi)}^{N} \iint_{\eabs{(x,\eta)} > \eabs{(t,\xi)}^{1/2}} \eabs{(x,\eta)}^{N-L} dx d\eta \\
& = C_L \eabs{(t,\xi)}^{N} \iint_{\eabs{(x,\eta)} > \eabs{(t,\xi)}^{1/2}} (1+|x|^2+|\eta|^2)^{(N-L)/2} dx d\eta \\
& = C_L \eabs{(t,\xi)}^{N} \int_{r > (\eabs{(t,\xi)}-1)^{1/2}} (1+r^2)^{(N-L)/2} r^{2d-1} dr \\
& \leq C_L \eabs{(t,\xi)}^{N} \int_{r > (\eabs{(t,\xi)}-1)^{1/2}} r^{N-L+2d-1} dr \\
& \leq C_L \eabs{(t,\xi)}^{3N/2-L/2+d}
\end{aligned}
\end{equation}
provided that $L>N+2d$ and $\eabs{(t,\xi)}$ is sufficiently large.
Since $L>0$ is arbitrary, \eqref{intuppsk1a} and \eqref{intuppsk2a} prove \eqref{decaycone1}, that is $g * \varphi(t,\xi)$ decays rapidly (polynomially) for $(t,\xi) \in \Gamma'$.
\end{proof}

\begin{rem}\label{permutvar1}
Obviously Lemma \ref{convolutiondecay1} is invariant under a change of roles of the variables, that is, with cones $\Gamma,\Gamma'$ of the form $\Gamma = \{(t,\xi): |t|>C|\xi| \}$, $C>0$.
\end{rem}

We are now prepared to define the following linear subspace of $\mathscr S'(\rr d)$.

\begin{defn}
Let $\varphi \in \mathscr S(\rr d) \setminus 0$.
\begin{equation}\label{conedef1}
\begin{aligned}
& V_{\rm con}^{\varphi} (\rr d) = \\
& \{ f \in \mathscr S'(\rr d): \ \exists B>0 : \sup_{|\eta|>B|x|} \eabs{ (x,\eta) }^n |V_\varphi f(x,\eta)| < \infty \quad \forall n>0 \}.
\end{aligned}
\end{equation}
\end{defn}
According to this definition $V_{\rm con}^{\varphi} (\rr d)$ consists of tempered distributions such that the STFT decays rapidly in some conical neighborhood of the frequency axis of the form $\{ (x,\eta) \in \rr {2d}: \ |\eta|>B|x| \}$. Obviously $V_{\rm con}^{\varphi} (\rr d)$ is a linear space. A priori $V_{\rm con}^{\varphi} (\rr d)$ depends on the window function $\varphi \in \mathscr S(\rr d)$, but the next lemma shows that this is in fact not the case.

\begin{lem}\label{coneindependenttest1}
If $f \in V_{\rm con}^{\varphi} (\rr d)$ then $f \in V_{\rm con}^{\psi} (\rr d)$ for any $\psi \in \mathscr S(\rr d) \setminus 0$.
\end{lem}
\begin{proof}
Let $f \in V_{\rm con}^{\varphi} (\rr d)$.
According to \cite[Lemma 11.3.3]{Grochenig1} we have
\begin{equation}\nonumber
|V_\psi f(t,\xi)| \leq C |V_\varphi f| * |V_\psi \varphi| (t,\xi), \quad t,\xi \in \rr d.
\end{equation}
Furthermore, we have $V_\psi \varphi \in \mathscr S(\rr {2d})$, $V_\varphi f \in C^\infty(\rr {2d})$ and by \cite[Theorem 11.2.3]{Grochenig1} there exist $C,N>0$ such that
$$
|V_\varphi f(t,\xi)| \leq C \eabs{(t,\xi)}^N, \quad t,\xi \in \rr d.
$$
The result now follows from Lemma \ref{convolutiondecay1}.
\end{proof}

As a consequence of the latter lemma we may denote $V_{\rm con}^{\varphi} (\rr d)=V_{\rm con} (\rr d)$, which is understood to be defined by an arbitrary $\varphi \in \mathscr S(\rr d) \setminus 0$.
Another consequence is that if \eqref{conedef1} holds, that is $V_\varphi f(x,\eta)$ decays rapidly in a conical neighborhood of the frequency axis $\{ (x,\eta) \in \rr {2d}: \ |\eta|>B|x| \}$, then rapid decay holds for a neighborhood of the form $\{ (x,\eta) \in \rr {2d}: |\eta|>B'|x-y| \}$, $B'>B$, for any fixed $y \in \rr d$. In fact we have $V_{T_{-y} \varphi} f(x,\eta) = V_\varphi f(x-y,\eta)$ which gives
\begin{equation}\nonumber
\begin{aligned}
& \sup_{(x,\eta): |\eta| > B'|x-y|} \eabs{(x,\eta)}^n |V_\varphi f(x,\eta)| \\
& = \sup_{|\eta| > B'|x-y|} \eabs{(x-y,\eta) + (y,0)}^n |V_{T_y \varphi} f(x-y,\eta)| \\
& \leq C \eabs{y}^n \sup_{|\eta| > B'|x-y|} \eabs{(x-y,\eta)}^n |V_{T_y \varphi} f(x-y,\eta)|
< \infty \quad \forall n>0
\end{aligned}
\end{equation}
by the proof of Lemma \ref{coneindependenttest1} and Lemma \ref{convolutiondecay1}.

\begin{prop}\label{smoothslowinclusion1}
We have the inclusions
\begin{align}
& C_{\rm slow}^\infty (\rr d) \subseteq WFW^{\perp}(\rr d), \label{inclusion1} \\
& C_{\rm slow}^\infty (\rr d) \subseteq V_{\rm con} (\rr d) \subseteq WFW^{\neq}(\rr d). \label{inclusion2}
\end{align}
\end{prop}
\begin{proof}
Let $f \in C_{\rm slow}^\infty (\rr d)$ and $\varphi \in \mathscr S(\rr d)$. For $\alpha \in \nn d$ integration by parts and \eqref{conditionf} give for some $N>0$
\begin{equation}\label{decaycone2}
\begin{aligned}
| \eta^\alpha V_\varphi f(x,\eta)| & = \left| C_\alpha \int_{\rr d} f(t) \overline{\varphi(t-x)} \pdd t \alpha (e^{-2 \pi i t \cdot \eta}) dt \right| \\
& = \left| \sum_{\beta \leq \alpha} C_\beta \int_{\rr d} \partial^{\alpha-\beta } f(t) \overline{\pd \beta \varphi(t-x)} e^{-2 \pi i t \cdot \eta} dt \right| \\
& \leq \sum_{\beta \leq \alpha} C_\beta \int_{\rr d} \eabs{t}^N \eabs{t-x}^{-N-d-1} dt \\
& \leq \sum_{\beta \leq \alpha} C_\beta \int_{\rr d} \eabs{t}^N \eabs{t}^{-N-d-1} \eabs{x}^{N+d+1} dt \\
& \leq C_\alpha \eabs{x}^{N+d+1}.
\end{aligned}
\end{equation}
Let $B>0$ be arbitrary and define the cone $\Gamma=\{(x,\eta) \in \rr {2d}: \ |\eta|>B|x| \}$.
Then if $(x,\eta) \in \Gamma$, we have by \eqref{decaycone2} for any $n>0$
\begin{equation}\nonumber
\eabs{\eta}^n | V_\varphi f(x,\eta)| \leq C_n \eabs{x}^{N+d+1} \leq C_n \eabs{\eta}^{N+d+1},
\end{equation}
which leads to
\begin{equation}\nonumber
\sup_{ (x,\eta) \in \Gamma}\eabs{(x,\eta)}^n | V_\varphi f(x,\eta)| \leq C \sup_{ (x,\eta) \in \Gamma} \eabs{\eta}^n | V_\varphi f(x,\eta)| < \infty
\end{equation}
for any $n>0$. Thus $V_\varphi f(x,\eta)$ decays rapidly in a cone $|\eta|>B|x|$ for any $B>0$. This proves the first inclusion in \eqref{inclusion2}.

By Lemma \ref{wignerfourier1} we have
\begin{equation}\label{rapiddecaycone1}
\sup_{|\eta|>B|x|} \eabs{(x,\eta)}^n | \wh {W_f W_\varphi} (\eta,x) | < \infty \quad \forall n>0
\end{equation}
for any $B>0$.
Let $\chi \in C_c^\infty(\rr {2d})$ and $\varphi \in \mathscr S(\rr d)$ satisfy $(\chi W_\varphi) (t,\xi) \neq 0$ for arbitrary fixed $(t,\xi) \in \rr {2d}$. We have $\chi W_\varphi \in C_c^\infty(\rr {2d})$, and
\begin{equation}\label{convolution1}
\mathscr F (W_f W_\varphi \chi) (\eta,x) = \wh {W_f W_\varphi} * \wh \chi (\eta,x).
\end{equation}
We have $\wh {W_f W_\varphi} = \wh W_f * \wh W_\varphi \in C^\infty(\rr {2d})$ and $\wh {W_f W_\varphi}$ has polynomial growth by \cite[Theorem 7.19]{Rudin2} since $\wh W_f \in \mathscr S'(\rr {2d})$ and $\wh W_\varphi \in \mathscr S(\rr {2d})$. This means that there exist $C,N>0$ such that
$$
| \wh {W_f W_\varphi} (\eta,x)| \leq C \eabs{(\eta,x)}^N, \quad \eta,x \in \rr d.
$$
Now \eqref{rapiddecaycone1}, \eqref{convolution1}, Lemma \ref{convolutiondecay1} and Remark \ref{permutvar1} says that for any $B'>B$ and corresponding cone $\Gamma'=\{ |\eta| > B' |x| \} \subset \Gamma$, it holds that $\mathscr F (W_f W_\varphi \chi) (\eta,x)$ decays rapidly in the cone $\Gamma'$. Hence $\mathscr F (W_f W_\varphi \chi) (\eta,x)$ decays rapidly in a cone $|\eta|>B|x|$ for any $B>0$. Therefore $f \in WFW^\perp(\rr d)$. The inclusion \eqref{inclusion1} is therefore proved.

It remains to prove the second inclusion in \eqref{inclusion2}. Let $f \in V_{\rm con} (\rr d)$. This assumption and Lemma \ref{wignerfourier1} imply that $\wh {W_f W_\varphi}(\eta,x)$ decays rapidly in a conic neighborhood $|\eta|>B|x|$ for some $B>0$ for any $\varphi \in \mathscr S(\rr d)$. Hence by Lemma \ref{convolutiondecay1}, $\mathscr F(W_f W_\varphi \chi) (\eta,x) = \wh {W_f W_\varphi} * \wh \chi (\eta,x)$ decays rapidly in a conical neighborhood of the form $|\eta|>B'|x|$ where $B'>0$, for any $\chi \in C_c^\infty(\rr {2d})$. This means that $(t,\xi;\eta,0) \notin WF(W_f)$ for any $t,\xi \in \rr d$ and $\eta \in \rr d\setminus 0$, that is $f \in WFW^{\neq}(\rr d)$.
\end{proof}

\begin{rem}
A combination of Lemma \ref{fouriercompact1} and Proposition \ref{smoothslowinclusion1}, \eqref{inclusion1}, gives $\mathscr {F E}'(\rr d) \subseteq WFW^{\perp}(\rr d)$. This admits the interpretation that the very high (analytic) regularity of $f \in \mathscr {F E}'(\rr d)$ is reflected in the fact that $W_f$ inherits smoothness in the time direction. In fact the wave front set of $W_f$ is directed purely in the frequency direction and has no component in the time direction.
\end{rem}

\begin{rem}
We note that an alternative proof of the inclusion
$$
V_{\rm con} (\rr d) \subseteq WFW^{\neq}(\rr d)
$$
can be deduced from the results by Toft \cite{Toft1,Toft2}. In fact, \cite[Proposition 1.8 and Theorem 4.1]{Toft2}
imply the following result.
\begin{equation}\nonumber
\begin{aligned}
& \mbox{If $f \in \mathscr S'(\rr d)$ and $(0,0; \eta, x) \notin WF(A_f)$} \\
& \mbox{then $(t,\xi; \eta, 2 \pi x) \notin WF(W_f)$ for all $(t,\xi) \in \rr {2d}$}.
\end{aligned}
\end{equation}
(Note the factor $2 \pi$ that appears in front of $x$, due to different normalizations of the Wigner distribution.) Here $A_f$ is the so called \emph{ambiguity function} of $f$ \cite{Folland1,Grochenig1}, normalized for $f \in \mathscr S(\rr d)$ as
$$
A_f(\eta,x) = \left( \frac{2}{\pi} \right)^{d/2} \int_{\rr d} f(t-\eta) \overline{f(t + \eta)} e^{2 i t \cdot x} dt.
$$
One can show with computations resembling those of the proof of Lemma \ref{wignerfourier1} that
\begin{equation}\nonumber
\wh {A_f A_\varphi} (\eta,x) = |V_{\check \varphi} f(\pi x,-\pi \eta)|^2, \quad \eta,x \in \rr d,
\end{equation}
for $f \in \mathscr S'(\rr d)$, $\varphi \in \mathscr S(\rr d)$. Here $\check \varphi(t)=\varphi(-t)$. Using  Lemma \ref{convolutiondecay1} as in the proof of Proposition \ref{smoothslowinclusion1} with $(\chi A_\varphi)(0,0) \neq 0$
we obtain $(0,0; \eta,0) \notin WF(A_f)$ for $\eta \in \rr d \setminus 0$ provided $f \in V_{\rm con}(\rr d)$. An application of Toft's results finally gives $V_{\rm con} (\rr d) \subseteq WFW^{\neq}(\rr d)$.
\end{rem}

We may summarize the inclusions we have found as follows.
\begin{equation}\nonumber
\begin{aligned}
\mathscr F \mathscr E'(\rr d) \subseteq & \ C_{\rm slow}^\infty(\rr d) \subseteq V_{\rm con} (\rr d) \subseteq WFW^{\neq}(\rr d), \\
& \ C_{\rm slow}^\infty(\rr d) \subseteq WFW^{\perp}(\rr d) \subseteq WFW^{\neq}(\rr d).
\end{aligned}
\end{equation}

The next example shows that $V_{\rm con} (\rr d) \nsubseteq WFW^{\perp}(\rr d)$ and $C_{\rm slow}^\infty (\rr d) \subsetneq V_{\rm con} (\rr d)$.

\begin{example}
Let $\eta_0 \in \ro \setminus 0$ and set $f(t) = \exp(2 \pi i \eta_0 t^2)$. Then by Example \ref{example1} we have $WF(W_f) = \{ (t, 2 \eta_0 t; -2 \eta_0 x,x), \ t \in \rr d,\ x \in \rr d \setminus 0 \}$, so $f \notin WFW^\perp(\rr d)$. If we choose $\varphi(x)=\exp(- 2 \pi x^2)$ then it can be verified that
\begin{equation}\nonumber
| V_\varphi f(x,\eta)| = C_{\eta_0} \exp \left( - \frac{2 \pi}{1+\eta_0^2} \left( \eta_0 x - \eta/2 \right)^2 \right), \quad C_{\eta_0} > 0.
\end{equation}
If $B>0$, $|\eta| >B |x|$ and $0 < \ep < 1/2$ we have
\begin{equation}\nonumber
\left| \eta_0 x - \frac{\eta}{2} \right| \geq \frac{|\eta|}{2} - | \eta_0| |x| > |\eta|\left(\frac1{2} - \frac{| \eta_0|}{B} \right) \geq \ep |\eta|
\end{equation}
provided $B$ is sufficiently large. Thus
\begin{equation}\nonumber
\begin{aligned}
\sup_{|\eta|>B |x|} \eabs{(x,\eta)}^n | V_\varphi f(x,\eta)|
& \leq C \sup_{|\eta|>B |x|} \eabs{\eta}^n | V_\varphi f(x,\eta)| \\
& \leq C \sup_{|\eta|>B |x|} \eabs{\eta}^n \exp \left( - \frac{2 \pi \ep^2}{1+\eta_0^2} |\eta|^2 \right) < \infty,
\end{aligned}
\end{equation}
which means that $f \in V_{\rm con} (\rr d)$. Hence we have shown $V_{\rm con} (\rr d) \nsubseteq WFW^{\perp}(\rr d)$.
For the function $f$ we also have $f \notin C_{\rm slow}^\infty (\rr d)$ which shows that
$C_{\rm slow}^\infty (\rr d) \subsetneq V_{\rm con} (\rr d)$. In fact, we have
$$
|\pd \alpha f(t)| = |p_\alpha(t)|
$$
where $p_\alpha$ is a polynomial of order $|\alpha|$. If we suppose that $f \in C_{\rm slow}^\infty (\rr d)$ then \eqref{conditionf} gives for some $N>0$
$$
\sup_{t \in \rr d} |\pd \alpha f(t)| \eabs{t}^{-N} = \sup_{t \in \rr d}|p_\alpha (t)|\eabs{t}^{-N} < \infty
$$
for all $\alpha \in \nn d$, which is a contradiction. Thus $f \notin C_{\rm slow}^\infty (\rr d)$.
We note furthermore that a cone $|\eta| > B |x|$ where $V_\varphi f(x,\eta)$ decays rapidly has $B \geq 2 |\eta_0|$.
Therefore it is not always the case that the cone of decay for elements in $V_{\rm con} (\rr d)$ can be arbitrarily large, that is, $B$ arbitrarily small.
\end{example}

Let $A \in \rr {d \times d}$ be symmetric. The transformation $T_A$ defined by
$$
T_A f(t) = \exp(\pi i t \cdot A t) f(t)
$$
is continuous on $\mathscr S(\rr d)$ and extends to continuous transformation on $\mathscr S'(\rr d)$. The Wigner distribution is transformed according to
\begin{equation}\label{chirpwigner1}
W_{T_A f} (t,\xi) = W_f(t,\xi-A t).
\end{equation}
The next result treats the invariance and noninvariance under $T_A$ of four of the spaces introduced above. However, for $V_{\rm con} (\rr d)$ we cannot prove or disprove invariance.

\begin{prop}
Let $A \in \rr {d \times d}$ be symmetric.

\rm{(i)} $\mathscr {F E'}(\rr d)$, $C_{\rm slow}^\infty(\rr d)$ and $WTW^\perp(\rr d)$ are not invariant under $T_A$.

\rm{(ii)} $WTW^{\neq}(\rr d)$ is invariant under $T_A$.
\end{prop}
\begin{proof}
(i) Since $1 \in \mathscr {F E'}(\rr d) \subseteq C_{\rm slow}^\infty(\rr d) \subseteq WTW^\perp(\rr d)$ by Lemma \ref{fouriercompact1} and Proposition \ref{smoothslowinclusion1}, it suffices to show that $g \notin WTW^\perp(\rr d)$ where $g(t) = \exp(\pi i t \cdot A t)$. According to Example \ref{example1} we have
$$
WF(W_g) = \{ (t,At;-Ax,x): \ t \in \rr d, \ x \in \rr d \setminus 0 \}.
$$
Picking $x \in \rr d$ such that $Ax \neq 0$ reveals that $g \notin WTW^\perp(\rr d)$.

(ii) According to \eqref{chirpwigner1} we have $W_{T_A f} = W_f \circ Q$ for $f \in \mathscr S'(\rr d)$ where $Q$ denotes the invertible matrix
\begin{equation}\nonumber
Q =
\left(
  \begin{array}{cc}
    I & 0 \\
    -A & I
  \end{array}
\right) \in \rr {2d \times 2d}.
\end{equation}
According to \cite[Theorem 8.2.4]{Hormander1} we have
\begin{equation}\nonumber
\begin{aligned}
& WF( W_f \circ Q ) \\
& = \{ (t,\xi; Q^t(\eta,x) ): \ ( Q(t,\xi);\eta,x ) \in WF(W_f), \ t,\xi, \eta, x \in \rr d \} \\
& = \{ (t,\xi; \eta-Ax,x ): \ ( t,\xi-At;\eta,x ) \in WF(W_f), \ t,\xi, \eta, x \in \rr d \}.
\end{aligned}
\end{equation}
If $f \in WFW^{\neq} (\rr d)$ and $(t,\xi; \eta,x) \in WF(W_f)$ then $x \neq 0$, which implies that
$$
WF(W_{T_A f}) \bigcap \left(\rr {2d} \times (\rr d \setminus 0 \times 0)\right) = \emptyset.
$$
Hence $T_A f \in WFW^{\neq} (\rr d)$.
\end{proof}

\subsection{An elaboration of the inclusion \eqref{inclusion1}}

In this subsection we prove a result related to the inclusion $C_{\rm slow}^\infty (\rr d) \subseteq WFW^{\perp}(\rr d)$. We will prove a stronger statement under a stronger hypothesis. The conclusion again includes the fact that the wave front set is directed purely in the frequency direction. On top of that we add the statement that the location of the singular support is included in $\rr d \times 0$.

We begin by recalling the definition of two classes of symbols: the
H\"ormander classes $S^m_{\rho,\delta}(X\times\rr N)$ \cite{Hormander1} and the Shubin classes $S^m_\rho(\rr d)$ \cite{Shubin1}.

\begin{defn}
Let $X\subseteq \rr n_z$ be open, $m\in \ro$, $\rho\in(0,1]$,
and $\delta\in[0,1)$. Then $S^m_{\rho,\delta}(X\times \rr d_\tau)$ is
the subspace of $a \in C^\infty(X\times \rr d_\tau)$ such
that for every $\alpha \in \nn d$, $\beta \in \nn n$ and every compact
$K\subset X$ there exists a constant $C_{m,\alpha,\beta,K}>0$ so
that
\begin{equation*}
|\partial^\beta_z\partial^\alpha_\tau a(z,\tau)|\le
C_{m,\alpha,\beta,K}\langle \tau
\rangle^{m-\rho|\alpha|+\delta|\beta|}
\end{equation*}
is satisfied for $z\in K, \tau \in \rr d$.
\end{defn}

\begin{defn}\label{shubinclasses1}
For $m\in \ro$, $0 < \rho \leq 1$, $\Gamma^m_\rho(\rr d)$ is the subspace of
$f \in C^\infty(\rr d)$ such that for every
$\gamma \in \nn d$ there exists a constant $C_{m,\gamma}>0$ so that
\begin{equation*}
|\partial^\gamma_t f(t)|\le C_{m,\gamma}\langle t
\rangle^{m-\rho|\gamma|}
\end{equation*}
is satisfied for every $t\in \rr d$.
\end{defn}

\begin{rem}$S^m_{\rho,\delta}(X\times \rr d)$ and $\Gamma^m_{\rho}(\rr
d)$ are Fr{\' e}chet spaces with respect to the best constants appearing
in the estimates. If we extend Definition \ref{shubinclasses1} to $\rho=0$ we note that
$$
\Gamma_\rho^m(\rr d) \subseteq C_{\rm slow}^\infty(\rr d) = \bigcup_{m >0 } \Gamma_0^m(\rr d), \quad 0 < \rho \leq 1.
$$
\end{rem}

We set $\Gamma=\rr n_z\times (\rr d_\tau \setminus 0)$. A {\it
phase function} $\phi$ on $\Gamma$ is a real-valued smooth function that satisfies
the conditions:

(i) $\phi(z,\lambda\tau)=\lambda\phi(z,\tau)$ for $(z,\tau)\in
\Gamma$, $\lambda>0$.

(ii) $\nabla \phi(z,\tau)\ne 0$ for every $(z,\tau)\in \Gamma$.
\\
We recall the meaning of {\it oscillatory integrals} of the type
$$
I_\phi^a=\int_{\rr n}e^{i\phi(\cdot,\tau)}a(\cdot,\tau)\, d\tau
$$
where $a\in S^m_{\rho,\delta}(\rr {n+d})$ and $\phi$ is a phase
function on $\Gamma$ (cf. \cite[Theorem 7.8.2]{Hormander1}).
\begin{prop}
{\rm (i)} For fixed $u\in C^\infty_c(\rr n)$, and fixed phase function
$\phi$, the map defined by the absolutely convergent integral
\begin{equation}
C^\infty_c(\rr {n+d}) \ni a \longrightarrow I_{\phi}^a(u)=\int_{\rr
{n+d}}e^{i\phi(z,\tau)}a(z,\tau)u(z)\, dz\, d\tau
\end{equation}
has a unique extension to a continuous functional on
$S^m_{\rho,\delta}(\rr {n+d})$ for $m\in \ro$, $\rho\in(0,1]$,
$\delta\in[0,1)$. Hence $I_{\phi}^a(u)$ is well defined for every
$u\in C^\infty_c(\rr n)$ and $a\in S^m_{\rho,\delta}(\rr {n+d})$.

{\rm (ii)} For fixed $a\in S^m_{\rho,\delta}(\rr {n+d})$ the map
(well-defined from (i)):
\begin{equation}
C^\infty_c(\rr n) \ni u \longrightarrow I_{\phi}^a(u)
\end{equation}
is a distribution in ${\mathscr D'}(\rr n)$ of finite order, which
is indicated by
$$
I_\phi^a(\cdot)=\int_{\rr d}e^{i\phi(\cdot,\tau)}a(\cdot,\tau)\, d\tau$$ and
called an oscillatory integral of symbol $a$ and phase
$\phi$.
\end{prop}

A general result on the wave front set of oscillatory integrals is \cite[Theorem 8.1.9]{Hormander1}, which follows.
\begin{prop}\label{WF}
For the distribution $I_\phi^a$ we have the inclusion
\begin{equation}\label{WFFI}
WF(I_\phi^a) \subseteq \{(z,\nabla_z\phi(z,\tau)):
\nabla_\tau\phi(z,\tau)=0 \}.
\end{equation}
\end{prop}

Now we use the previous result in the study of the wave front set of
the Wigner distribution. More precisely, in the following proposition
we prove that, considering a function $f$ in a Shubin class, the
wave front set of the Wigner distribution $W_f$ is not only ``vertical'' in
the dual variables $(\eta,x)$, but the singular support of $W_f$ is also contained in the
``horizontal'' subspace $\xi=0$ of the space variables $(t,\xi)$.

\begin{prop}
If $\rho>0$ and $f\in \Gamma^m_\rho(\rr d)$ then
\begin{equation}\label{inclusion3}
WF(W_f)\subseteq\{(t,0;0,x)\in \rr {4d}: t\in \rr d, x\in\rr
d \setminus 0 \}.
\end{equation}
\end{prop}
\begin{proof}
If $f\in \Gamma^m_\rho(\rr d)$ we have
$f \otimes \overline{f} \circ \kappa \in S^{2m}_{\rho,0}(\rr {2d})$. In
fact, for every $\alpha, \beta \in \nn d$ and for $t$ in a
compact set $K\subset \rr d$, we have:
\begin{align*}
& |\partial^\alpha_\tau \partial^\beta_t
f(t+\tau/2)\overline{f(t-\tau/2)}|
\\
& =\left|\partial^\alpha_\tau \sum_{\beta'\le
\beta}c_{\beta,\beta'} \partial^{\beta'}f (t+\tau/2)
\overline{ \partial^{\beta-\beta'}f (t-\tau/2)} \right|
\\
& =\left|\sum_{\alpha'\le \alpha} \sum_{\beta'\le
\beta}c_{\alpha,\alpha'}c_{\beta,\beta'} \partial^{\alpha'+\beta'}f (t+\tau/2)
\overline{ \partial^{\alpha-\alpha'+\beta-\beta'}f (t-\tau/2)}
\right|
\\
& \le C \sum_{\alpha'\le \alpha} \sum_{\beta'\le
\beta} \langle
t+\tau/2\rangle^{m-\rho|\alpha'+\beta'|}\langle
t-\tau/2\rangle^{m-\rho|\alpha-\alpha'+\beta-\beta'|}
\\
& \le C \sum_{\alpha'\le \alpha} \sum_{\beta'\le \beta}
\langle t
\rangle^{|m-\rho|\alpha'+\beta'||+|m-\rho|\alpha-\alpha'+\beta-\beta'||}
\langle \tau
\rangle^{m-\rho|\alpha'+\beta'|+m-\rho|\alpha-\alpha'+\beta-\beta'|}
\\
& \le C \langle \tau \rangle^{2m-\rho|\alpha+\beta|}\le C \langle
\tau \rangle^{2m-\rho|\alpha|}
\end{align*}
where Petree's inequality $\langle t \pm \tau/2\rangle^s\le C_s \langle
t\rangle^{|s|}\langle \tau \rangle^s$,\  $s\in \ro$, \ has been
used.

We set now $z=(t,\xi)\in\rr d_t\times \rr d_\xi$ and remark that a
symbol $a(t,\tau)\in S^m_{\rho,\delta}(\rr d_t\times \rr d_\tau)$ can
also be seen as a symbol in $S^m_{\rho,\delta}(\rr {2d}_z\times \rr
d_\tau)$, and likewise a phase function $\phi(t,\tau)$ on $\rr
d_t\times \rr d_\tau$ is also a phase function on $\rr {2d}_z\times
\rr d_\tau$. We can therefore apply Proposition \ref{WF} in the case
$n=2d$ with
$$
a(t,\xi,\tau) := f(t+\tau/2)\overline{f(t-\tau/2)}\in
S^{2m}_{\rho,0}(\rr {2d}_{t,\xi}\times \rr d_\tau)
$$
and phase function on $\rr {2d}_{t,\xi} \times \rr d_\xi$
$$
\phi(t,\xi,\tau)=-2 \pi \xi\cdot\tau.
$$
Thus we have $W_f = I^a_\phi$, and
\begin{align*}
& \nabla_z \phi=(0,-2 \pi \tau)\in \rr {2d}, \\
& \nabla_\tau\phi=-2 \pi \xi\in \rr d,
\end{align*}
which means that \eqref{WFFI} implies the inclusion \eqref{inclusion3}.
\end{proof}

\section{Restriction of the Wigner distribution to fixed time}\label{restriction}

In this section we shall study the restriction operator of a distribution $F \in \mathscr D'(\rr {2d})$ to the submanifold $t \times \rr d \subseteq \rr {2d}$ for $t \in \rr d$ fixed, which is denoted
\begin{equation}\label{restriction1}
R_t F = F(t,\cdot).
\end{equation}
This map is not well defined for any $F \in \mathscr D'(\rr {2d})$.
But according to \cite[Corollary 8.2.7]{Hormander1}, the restriction \eqref{restriction1} gives a well defined element in $\mathscr D'(\rr d)$ provided $WF(F) \cap N_t = \emptyset$, where
\begin{equation}\label{normalbundle1}
N_t = ( t \times \rr d)  \times (\rr d \times 0)
\end{equation}
is the normal bundle of the submanifold $t \times \rr d \subseteq \rr {2d}$.
For $f \in WFW^{\neq}(\rr {d})$ we have $WF(W_f) \cap N_t = \emptyset$ for any $t \in \rr d$, so \cite[Corollary 8.2.7]{Hormander1} implies that the restriction $W_f(t,\cdot)$ gives an element in $\mathscr D'(\rr d)$ for any $t \in \rr d$.

The distribution $f = \delta_0 = \delta_0(\rr d)$ has Wigner distribution $W_f = \delta_0 \otimes 1$ \cite{Grochenig1}, which has wave front set
$$
WF(W_f) = (0 \times \rr d) \times ( \rr d \setminus 0 \times 0).
$$
This example shows that $WF(W_f) \cap N_t = \emptyset$ is not always satisfied.

We will study the continuity properties of the restriction \eqref{restriction1}, and for that purpose
we need the following definitions and results from \cite{Hormander1}. For a closed set $\Gamma \subseteq \rr d \times (\rr d \setminus 0)$, conic in the second variable, we define
$$
\mathscr D_\Gamma '(\rr d) = \{ u \in \mathscr D '(\rr d), \ WF(u) \subseteq \Gamma \}.
$$
If $u, u_j \in \mathscr D_\Gamma '(\rr d)$ for $1 \leq j < \infty$ we say that $u_j \rightarrow u$ in $\mathscr D_\Gamma '(\rr d)$ provided
\begin{align}
& u_j \longrightarrow u \quad \mbox{in} \quad \mathscr D '(\rr d), \quad \mbox{and} \label{konkonv1} \\
& \sup_{ \xi \in V} |\xi|^n |\mathscr F( \varphi (u-u_j) )(\xi)| \rightarrow 0, \quad j \rightarrow \infty, \quad \forall n \in \no, \label{konkonv2}
\end{align}
for a closed conic set $V \subseteq \rr d$ and $\varphi \in C_c^\infty(\rr d)$ such that
$$
\Gamma \cap(\supp \varphi \times V) = \emptyset.
$$
It is sometimes more convenient to use instead of \eqref{konkonv2} the equivalent requirement \cite{Hormander1}
\begin{equation}\tag*{(\ref{konkonv2})$'$}
\sup_j \sup_{ \xi \in V} |\xi|^n |\wh{ \varphi u_j }(\xi)| < \infty \quad \forall n \in \no.
\end{equation}
For a closed set $\Gamma \subseteq \rr {2d} \times (\rr {2d} \setminus 0)$, conic in the second variable, we set
\begin{equation}\nonumber
\Gamma_t = \{ (\xi,x) \in \rr {2d}: \exists \eta \in \rr d: (t,\xi;\eta,x) \in \Gamma \}.
\end{equation}\
We will use \cite[Theorem 8.2.4]{Hormander1}.
Let $\Gamma \subseteq \rr {2d} \times (\rr {2d} \setminus 0)$ be a closed set, conic in the second variable, such that $\Gamma \cap N_t = \emptyset$. Then \cite[Theorem 8.2.4]{Hormander1} implies in particular that $R_t$ is a continuous map
\begin{equation}\label{Rcont1}
R_t: \mathscr D_\Gamma '(\rr {2d}) \mapsto \mathscr D_{\Gamma_t} '(\rr d).
\end{equation}

Next we study $F \in C_{\rm slow}^\infty(\rr {2d})$ and the following two operators:
(i) the restriction \eqref{restriction1} to the submanifold $t \times \rr d \subseteq \rr {2d}$ for $t \in \rr d$ fixed, and (ii) Fourier transformation in the second variable $F \mapsto \mathscr F_2 F$. The following results say that these two operators commute for $F \in C_{\rm slow}^\infty(\rr {2d})$. First we need a lemma.

\begin{lem}\label{restrictionlemma1}
Let $F \in C_{\rm slow}^\infty (\rr {2d})$, $\Gamma = \rr {2d} \times U$ where
\begin{equation}\label{Udef1}
U = \{ (\eta,x) \in \rr {2d} \setminus 0: |\eta| \leq |x| \},
\end{equation}
let $\chi \in C_c^\infty(\rr d)$, $\chi \geq 0$, $\chi(x)=1$ for $|x| \leq 1$, $\chi_j(x)=\chi(x/j)$, $j>0$ integer, and $F_j=F \chi_j \otimes \chi_j$. Then we have
\begin{equation}\label{partfourierkont1}
\mathscr F_2 F_j \longrightarrow \mathscr F_2 F \quad \mbox{in $\mathscr D_\Gamma'(\rr {2d})$ as $j \rightarrow \infty$}.
\end{equation}
\end{lem}
\begin{proof}
If $\mathscr F_2 \varphi \in C_c^\infty(\rr {2d})$ then $\varphi \in \mathscr S(\rr {2d})$ and we have, since $F_j \rightarrow F$ in $\mathscr S'(\rr {2d})$ as $j \rightarrow \infty$
\begin{equation}\nonumber
\begin{aligned}
( \mathscr F_2 F_j, \mathscr F_2 \varphi ) = (F_j,\varphi) \rightarrow (F,\varphi) = ( \mathscr F_2 F, \mathscr F_2 \varphi ), \quad j \rightarrow \infty.
\end{aligned}
\end{equation}
Thus $\mathscr F_2 F_j \rightarrow \mathscr F_2 F$ in $\mathscr D'(\rr {2d})$ as $j \rightarrow \infty$ and the first criterion \eqref{konkonv1} for \eqref{partfourierkont1} is proved. To prove the second criterion we use \eqref{konkonv2}$'$.
Since $\mathscr F^{-1} (g(-\eta-\cdot)) = M_{-\eta} \wh g$ for $g \in \mathscr S(\rr d)$ we have for $\varphi \in C_c^\infty (\rr {2d})$
\begin{equation}\nonumber
\begin{aligned}
\mathscr F (\varphi \mathscr F_2 F_j) (\eta,x) & =
\mathscr F^{-1} (\varphi \mathscr F_2 F_j) (-\eta,-x) \\
& = \mathscr F_1^{-1} F_j * \mathscr F^{-1} \varphi (-\eta,-x) \\
& = \langle \mathscr F_1^{-1} F_j, \mathscr F^{-1} \varphi (-\eta-\cdot,-x-\cdot) \rangle \\
& = \langle F_j, \mathscr F_1^{-1} \left( \mathscr F^{-1} \varphi (-\eta-\cdot,-x-\cdot) \right) \rangle \\
& = \langle F_j, M_{-\eta,0} \mathscr F_2^{-1} \varphi (\cdot,-x-\cdot) \rangle \\
& = \iint_{\rr {2d}} F(t,\tau) e^{-2 \pi i t \cdot \eta} \mathscr F_2^{-1} \varphi (t,-x-\tau) \chi_j(t) \chi_j(\tau) dt d\tau \\
& = \iint_{\rr {2d}} F(t,\tau) e^{-2 \pi i t \cdot \eta} \mathscr F_2^{-1} \varphi (t,-x-\tau) \chi_j(\tau) dt d\tau,
\end{aligned}
\end{equation}
for $j \geq J$ for sufficiently large $J$. In fact, the last equality holds for such $j$ since $\mathscr F_2^{-1} \varphi (\cdot,y)$ has support in a fixed compact set, independent of $y \in \rr d$. The assumption implies that
\begin{equation}\label{partderestimate1}
|\pdd t \alpha  F(t,\tau)| \leq C_\alpha \eabs{t}^{N} \eabs{\tau}^{N}, \quad \alpha \in \nn d,
\end{equation}
for some $C_\alpha, N>0$. By means of integration by parts, \eqref{partderestimate1} and the observation that $\mathscr F_2^{-1} \varphi \in \mathscr S$, we obtain for $j \geq J$ and any $\alpha \in \nn d$
\begin{equation}\nonumber
\begin{aligned}
& \left| \eta^\alpha \mathscr F (\varphi \mathscr F_2 F_j) (\eta,x) \right| \\
& = \left|  \iint_{\rr {2d}} (-2 \pi i)^{-|\alpha|} \pdd t \alpha \left( e^{-2 \pi i t \cdot \eta} \right) F(t,\tau) \mathscr F_2^{-1} \varphi (t,-x-\tau) \chi_j(\tau) dt d\tau \right| \\
& \leq \sum_{\beta \leq \alpha} C_\beta \iint_{\rr {2d}} \left| \pdd t \beta F(t,\tau) \right| \left| \partial_t^{\alpha-\beta} \mathscr F_2^{-1} \varphi (t,-x-\tau) \right| dt d\tau \\
& \leq \sum_{\beta \leq \alpha} C_\beta \iint_{\rr {2d}} \eabs{t}^{N} \eabs{\tau}^{N} \eabs{t}^{-N-d-1} \eabs{\tau+x}^{-N-d-1} dt d\tau \\
& \leq C_\alpha \sum_{\beta \leq \alpha} \eabs{x}^{N+d+1} \iint_{\rr {2d}} \eabs{t}^{-d-1} \eabs{\tau}^{N-N-d-1} dt d\tau \\
& \leq C_\alpha \eabs{x}^{N+d+1},
\end{aligned}
\end{equation}
for some constant $C_\alpha>0$. This gives for $j \geq J$
\begin{equation}\label{wavefrontcont1}
\begin{aligned}
& \left| \eta^\alpha \mathscr F (\varphi \mathscr F_2 F_j) (\eta,x) \right| \\
& = \eabs{\eta}^{-2 (N+d+1)} \left| \eabs{\eta}^{2 (N+d+1)} \eta^\alpha \mathscr F (\varphi \mathscr F_2 F_j) (\eta,x) \right| \\
& = \eabs{\eta}^{-2 (N+d+1)} \left| \sum_{|\gamma| \leq 2 (N+d+1)} C_\gamma \eta^{\alpha+\gamma} \mathscr F (\varphi \mathscr F_2 F_j) (\eta,x) \right| \\
& \leq C_\alpha \eabs{\eta}^{-2 (N+d+1)} \eabs{x}^{N+d+1}, \quad \alpha \in \nn d, \quad C_\alpha>0.
\end{aligned}
\end{equation}
Now let $V \subseteq \rr d$ be a closed conic set and let $\varphi \in C_c^\infty(\rr {2d}) \setminus 0$ satisfy
\begin{equation}\nonumber
\emptyset = \Gamma \cap(\supp \varphi \times V) = \supp \varphi \times (U \cap V),
\end{equation}
which by \eqref{Udef1} means that $(\eta,x) \in V \setminus 0 \Rightarrow |x| < |\eta|$.
Let $n \in \no \setminus 0$ and let $(\eta,x) \in V \setminus 0$. Then we have
$$
|(\eta,x)|^n = (|\eta|^2+|x|^2)^{n/2} \leq 2^{n/2} |\eta|^n \leq (2 d)^{n/2} \max_{\alpha: \ |\alpha|=n} |\eta^\alpha|.
$$
For $j \geq J$ \eqref{wavefrontcont1} thus implies
\begin{equation}\label{wavefrontestimate1}
\begin{aligned}
\sup_{(\eta,x) \in V} |(\eta,x)|^n \left| \mathscr F (\varphi \mathscr F_2 F_j) (\eta,x) \right|
& \leq C_n \sup_{(\eta,x) \in V} \sup_{|\alpha|=n} \left| \eta^\alpha \mathscr F (\varphi \mathscr F_2 F_j) (\eta,x) \right| \\
& \leq C_n \sup_{(\eta,x) \in V} \sup_{|\alpha|=n} C_\alpha \eabs{\eta}^{-2 (N+d+1)} \eabs{x}^{N+d+1} \\
& \leq C_n \sup_{(\eta,x) \in V} \sup_{|\alpha|=n} C_\alpha \eabs{\eta}^{-2 (N+d+1)} \eabs{\eta}^{N+d+1} \\
& \leq C_n
\end{aligned}
\end{equation}
for some $C_n>0$, independently of $j \geq J$. This means that the second criterion \eqref{konkonv2}$'$ for
the convergence \eqref{partfourierkont1} has been proved.

It remains to verify that $\mathscr F_2 F_j, \mathscr F_2 F \in \mathscr D_\Gamma '(\rr {2d})$. If $(\eta,x) \notin U$ and $(\eta,x) \neq 0$ then $|x| < |\eta|$ and there is an open conic neighborhood $U'$ containing $(\eta,x)$ of the form $U'=\{(\eta,x): \ |x| < C|\eta| \}$ for $C>0$ such that $U \cap U' = \emptyset$. The estimate \eqref{wavefrontestimate1} with $V$ replaved by $U'$ shows that
$WF( \mathscr F_2 F_j ) \subseteq \Gamma$. Likewise $WF( \mathscr F_2 F) \subseteq \Gamma$ so we have proved $\mathscr F_2 F_j, \mathscr F_2 F \in \mathscr D_\Gamma '(\rr {2d})$.
\end{proof}

\begin{prop}\label{restrictionprop1}
If $F \in C_{\rm slow}^\infty (\rr {2d})$ and $t \in \rr d$ then we have
\begin{equation}\label{commute1}
\mathscr F ( F(t,\cdot) ) = (\mathscr F_2 F)(t,\cdot) \quad \mbox{in $\mathscr D'(\rr d)$}.
\end{equation}
\end{prop}
\begin{proof}
The result \eqref{commute1} says that $\mathscr F \circ R_t = R_t \circ \mathscr F_2$ on $C_{\rm slow}^\infty(\rr {2d})$, where the equality is understood in $\mathscr D'(\rr d)$.
Let $U \subset \rr {2d} \setminus 0$ be defined by \eqref{Udef1}.
Then $\Gamma = \rr {2d} \times U$ is a closed set, conic in the second variable, and by \eqref{normalbundle1} we have
\begin{equation}\nonumber
\Gamma \cap N_t = ( t \times \rr d)  \times U \cap (\rr d \times 0) = \emptyset.
\end{equation}
Thus the restriction \eqref{restriction1} is a continuous map between distribution spaces according to \eqref{Rcont1}.
By Lemma \ref{restrictionlemma1} and \eqref{Rcont1} we have $(\mathscr F_2 F)(t,\cdot) \in \mathscr D'(\rr d)$.

Let $\chi \in C_c^\infty(\rr d)$, $\chi \geq 0$, $\chi(x)=1$ for $|x| \leq 1$, $\chi_j(x)=\chi(x/j)$ and $F_j=F \chi_j \otimes \chi_j \in C_c^\infty(\rr {2d})$. As before $F_j \rightarrow F$ in $\mathscr S'(\rr {2d})$, and therefore also in $\mathscr D'(\rr {2d})$, as $j \rightarrow \infty$. For any $\varphi \in C_c^\infty(\rr {2d})$ we have $\varphi(F_j-F)=\varphi (\chi_j \otimes \chi_j-1)F \equiv 0$ for $j$ sufficiently large, which together with $F_j \rightarrow F$ in $\mathscr D'(\rr {2d})$ imply that $F_j \rightarrow F$ in $\mathscr D_\Gamma '(\rr {2d})$, since $F_j,F \in \mathscr D_\Gamma '(\rr {2d})$ because $WF(F_j) = WF(F) = \emptyset$. Thus \eqref{Rcont1} and \eqref{konkonv1} implies that $F_j(t,\cdot) \rightarrow F(t,\cdot)$ in $\mathscr D'(\rr d)$. We have in fact
\begin{equation}\label{Rcont2}
F_j(t,\cdot) \rightarrow F(t,\cdot) \quad \mbox{in} \quad \mathscr S'(\rr {d}).
\end{equation}
To see this, let $\varphi \in \mathscr S(\rr d)$. Since $|F(t,\tau)| \leq C \eabs{t}^N \eabs{\tau}^N$ for some $C,N>0$ we have for any $\ep>0$
\begin{equation}\nonumber
|\langle F_j(t,\cdot),\varphi (\chi_k-1) \rangle | \leq 2 C \eabs{t}^N \int_{|\tau|\geq k} \eabs{\tau}^N |\varphi(\tau)| d\tau < \ep \quad \mbox{uniformly in $j$}
\end{equation}
for $k \geq k_1$ and $k_1$ sufficiently large ($t$ is fixed). Since $F(t,\cdot) \in \mathscr S'(\rr d)$ and $\varphi \chi_k \rightarrow \varphi$ in $\mathscr S(\rr d)$ as $k \rightarrow \infty$ we have
\begin{equation}\nonumber
|\langle F(t,\cdot),\varphi (1-\chi_k) \rangle| < \ep \quad \mbox{if $k \geq k_2$}
\end{equation}
for $k_2$ sufficiently large. The latter two estimates now give for $k \geq \max(k_1,k_2)$
\begin{equation}\nonumber
\begin{aligned}
& |\langle F(t,\cdot) - F_j(t,\cdot),\varphi \rangle | \\
& \leq | \langle F(t,\cdot),\varphi(1-\chi_k) \rangle | + | \langle F(t,\cdot)-F_j(t,\cdot),\varphi \chi_k \rangle| + | \langle F_j(t,\cdot),\varphi(\chi_k-1) \rangle | \\
& < 3 \ep
\end{aligned}
\end{equation}
for $j$ sufficiently large, since $F_j(t,\cdot) \rightarrow F(t,\cdot)$ in $\mathscr D'(\rr d)$. This proves \eqref{Rcont2}.

Now we use Lemma \ref{restrictionlemma1}. The convergence \eqref{partfourierkont1}, \eqref{Rcont1} and \eqref{konkonv1} give
\begin{equation}\nonumber
(\mathscr F_2 F_j)(t,\cdot) \longrightarrow (\mathscr F_2 F)(t,\cdot) \quad \mbox{in $\mathscr D'(\rr d)$ as $j \rightarrow \infty$}.
\end{equation}
For $\wh \varphi \in C_c^\infty(\rr d)$, this finally gives, using \eqref{Rcont2} and Fubini's theorem,
\begin{equation}\nonumber
\begin{aligned}
( \mathscr F(F(t,\cdot)), \wh \varphi ) & = (F(t,\cdot), \varphi ) \\
& = \lim_{j \rightarrow \infty} ( F_j(t,\cdot), \varphi ) \\
& = \lim_{j \rightarrow \infty} \int_{\rr d}  F_j(t,\tau) \overline{\varphi(\tau)} d \tau \\
& = \lim_{j \rightarrow \infty} \int_{\rr d}  F_j(t,\tau) \overline{ \left( \int_{\rr d} \wh \varphi(\xi) e^{2 \pi i \tau \cdot \xi} d \xi \right) } d \tau \\
& = \lim_{j \rightarrow \infty} \int_{\rr d}  (\mathscr F_2 F_j)(t,\xi) \overline{\wh \varphi(\xi)}d \xi \\
& = \lim_{j \rightarrow \infty} ( (\mathscr F_2 F_j)(t,\cdot), \wh \varphi ) \\
& = ( (\mathscr F_2 F)(t,\cdot), \wh \varphi ).
\end{aligned}
\end{equation}
\end{proof}

If $f \in C_{\rm slow}^\infty (\rr d)$ then $f \otimes \overline{f} \circ \kappa \in C_{\rm slow}^\infty (\rr {2d})$, so the definition of the Wigner distribution \eqref{wd1} combined with Proposition \ref{restrictionprop1} gives the following byproduct.

\begin{cor}\label{restrictioncor0}
If $f \in C_{\rm slow}^\infty (\rr d)$ and $t \in \rr d$ then we have
\begin{equation}\nonumber
\mathscr F ( f \otimes \overline{f} \circ \kappa (t,\cdot) ) = W_f(t,\cdot) \quad \mbox{in $\mathscr D'(\rr d)$}.
\end{equation}
\end{cor}

Finally Lemma \ref{fouriercompact1} gives

\begin{cor}\label{restrictioncor1}
If $f \in \mathscr{F E}'(\rr d)$ and $t \in \rr d$ then we have
\begin{equation}\nonumber
\mathscr F ( f \otimes \overline{f} \circ \kappa (t,\cdot) ) = W_f(t,\cdot) \quad \mbox{in $\mathscr D'(\rr d)$}.
\end{equation}
\end{cor}

\section{A Wigner distribution moment formula for the instantaneous frequency}\label{wignermoment}

Denote the modulus of $z =x+iy \in \co$ by $r=|z|$ and the argument (or phase) by $\varphi=\arg z$.
The polar-to-rectangular coordinate transformation on $\rr 2 \simeq \co$ is defined by
$(x,y) = g(r,\varphi) = (r \cos \varphi,r \sin \varphi)$.
Denote the positive reals by $\ro_+$.
Then each of the restrictions $g_1$ and $g_2$, defined respectively by
$$
g_1 := g|_{\ro_+ \times (-\pi,\pi)}, \quad g_2 := g|_{\ro_+ \times (0,2\pi)},
$$
is an analytic map with an analytic inverse, mapping surjectively on the open sets $U_1$ and $U_2$, respectively, defined by
\begin{equation}\nonumber
\begin{aligned}
& g_1: \ro_+ \times (-\pi,\pi) \mapsto \co \setminus \{ (-\infty,0] +i 0 \} := U_1, \\
& g_2: \ro_+ \times (0,2\pi) \mapsto \co \setminus \{ [0,+\infty) +i 0 \} := U_2.
\end{aligned}
\end{equation}
The Jacobian of $g$ is
\begin{equation}\nonumber
Dg = \left(
  \begin{array}{cc}
      \frac{\partial x}{\partial r} & \frac{\partial x}{\partial \varphi} \\
      \frac{\partial y}{\partial r} & \frac{\partial y}{\partial \varphi}
    \end{array}
  \right)
  = \left(
  \begin{array}{cc}
      \cos \varphi & -r \sin \varphi \\
      \sin \varphi & r \cos \varphi
    \end{array}
  \right)
\end{equation}
which is invertible with inverse
\begin{equation}\nonumber
(Dg)^{-1}
  = \frac1{r} \left(
  \begin{array}{cc}
      r \cos \varphi & r \sin \varphi \\
      - \sin \varphi & \cos \varphi
    \end{array}
  \right)
  = \frac1{x^2 + y^2} \left(
  \begin{array}{cc}
      x \sqrt{x^2 + y^2}  & y \sqrt{x^2 + y^2} \\
      - y & x
    \end{array}
  \right)
\end{equation}
provided $r>0$. The inverse function theorem applied to each of the restrictions $g_1$ and $g_2$ thus gives
\begin{equation}\label{partialderivatives1}
\frac{\partial \arg z}{\partial x} = \frac{-y}{x^2+y^2}, \quad \frac{\partial \arg z}{\partial y} = \frac{x}{x^2+y^2}, \quad x+iy \in U_1 \cup U_2.
\end{equation}
(Note that $g_1^{-1}(z)$ and $g_2^{-1}(z)$ have arguments that differ by $2 \pi$ for $\im z<0$, but this does not affect the partial derivatives of $\arg z$.)

Let $f(t)=u(t) + i v(t)$ be a function $f: \rr d \mapsto \co$. If $f$ is continuous then $U_{f,1} := \{ t \in \rr d: f(t) \in U_1 \}$ and $U_{f,2} := \{ t \in \rr d: f(t) \in U_2 \}$ are open sets, and it follows from above that each of $\arg f(t) = (g_1^{-1} \circ f (t))_2: U_{f,1} \mapsto \ro$ and $\arg f(t) = (g_2^{-1} \circ f (t))_2: U_{f,2} \mapsto \ro$ are continuous. Here $(g_1^{-1} \circ f (t))_2$ means the second component of $g_1^{-1} (f (t))$. If $f \in C^1(\rr d)$ then, likewise, each of $\arg f(t): U_{f,1} \mapsto \ro$ and $\arg f(t): U_{f,2} \mapsto \ro$ are differentiable. From \eqref{partialderivatives1} and the chain rule we obtain the partial derivative with respect to $t_j$, in both cases, as
\begin{equation}\label{fasderivata1}
\begin{aligned}
\frac{\partial \arg f(t)}{\partial t_j}
& = \frac{\partial \arg z}{\partial x} \Big|_{z=f(t)} \frac{\partial u(t)}{\partial t_j} + \frac{\partial \arg z}{\partial y} \Big|_{z=f(t)} \frac{\partial v(t)}{\partial t_j} \\
& = \frac{u(t) \partial_{t_j} v (t) - v(t) \partial_{t_j} u(t)}{u^2(t) + v^2(t)}
\end{aligned}
\end{equation}
for $t \in U_{f,1} \cup U_{f,2} = \{ t \in \rr d: \ f(t) \neq 0 \}$ and $1 \leq j \leq d$. Thus
$$
\nabla \arg f(t) = \frac{u(t) \nabla v (t) - v(t) \nabla u(t)}{u^2(t) + v^2(t)}, \quad f(t) \neq 0.
$$
The \emph{instantaneous frequency} \cite{Cohen1} of $f(t)$ is defined as the normalized gradient $(2 \pi)^{-1} \nabla \arg f(t)$ with domain $\{ t: f(t) \neq 0 \}$.
For a character $e^{2 \pi i \xi \cdot t}$ with frequency $\xi$ the instantaneous frequency is thus $\xi$ constantly, which means that the instantaneous frequency is a generalization of the concept constant (global) frequency.

There is a connection between the instantaneous frequency of a sufficiently smooth and decaying function and its Wigner distribution \cite{Cohen1,Flandrin1}. In fact, the instantaneous frequency can be written, for fixed $t \in \rr d$, as a normalized frequency moment of order one of the Wigner distribution. (In the engineering literature a heuristic version of this result is well known \cite{Cohen1}.)

\begin{prop}\label{ifwignerprop1}
Suppose $\ep>0$ and $f \in H^{\frac{d}{2}+1+\ep}(\rr d)$. Then for any $t \in \rr d$ such that $f(t) \neq 0$ we have
\begin{equation}\label{ifwigner1}
\frac1{2 \pi}\nabla \arg f(t) = \frac{\int_{\rr d} \xi W_f(t,\xi) d\xi }{ \int_{\rr d} W_f(t,\xi) d\xi }.
\end{equation}
\end{prop}
\begin{proof}
First we note that $f \in H^{\frac{d}{2}+1+\ep} \subseteq L^2$ and \eqref{wignercont1} imply that $W_f \in C_0(\rr {2d})$. Hence the restriction of the Wigner distribution $W_f \mapsto W_f(t,\cdot)$ is a well defined continuous map $C_0(\rr {2d}) \mapsto C_0(\rr d)$ for any $t \in \rr d$.

Let $t \in \rr d$ be fixed arbitrary such that $f(t) \neq 0$.
The assumption $f \in H^{\frac{d}{2}+1+\ep}(\rr d)$ and the Sobolev embedding theorem imply that
\begin{equation}\label{sobolevembedding1}
f, \ \partial_j f \in C_0(\rr d), \quad 1 \leq j \leq d.
\end{equation}
Cauchy--Schwarz and $f \in L^2$ give
\begin{equation}\label{finvsuff1}
g := f(t+ \cdot/2) \overline{f(t-\cdot/2)} \in L^1(\rr d).
\end{equation}
The assumption $f \in H^{\frac{d}{2}+1+\ep}(\rr d)$ also gives
\begin{equation}\label{sobolev2}
\begin{aligned}
& \int_{\rr d} \eabs{\xi} |\wh f(\xi)| d\xi
& \leq \left( \int_{\rr d} \eabs{\xi}^{2(\frac{d}{2}+1+\ep)} |\wh f(\xi)|^2 d\xi \right)^{1/2} \left( \int_{\rr d} \eabs{\xi}^{-d - 2 \ep } d \xi \right)^{1/2} < \infty.
\end{aligned}
\end{equation}
Hence $\wh f \in L^1(\rr d)$ and since $\mathscr F (f(t+ \cdot/2))(\xi) = 2^d M_{t} \wh f(2 \xi)$ we have $\mathscr F (f(t+ \cdot/2)) \in L^1(\rr d)$. Thus $\mathscr F (\overline{f(t- \cdot/2)})=\overline{\mathscr F (f(t+ \cdot/2))}$ implies that
\begin{equation}\label{finvsuff2}
\mathscr F g = \mathscr F (f(t+ \cdot/2)) * \overline{\mathscr F (f(t+ \cdot/2))} \in L^1(\rr d).
\end{equation}
By \eqref{sobolevembedding1}, \eqref{finvsuff1} and \eqref{finvsuff2} we have $g \in ( C_0 \cap L^1 \cap \mathscr F L^1)(\rr d)$, which means that
the conditions for Fourier's inversion formula to hold pointwise \cite{Rudin1} for the function $g$ are satisfied. In particular
\begin{equation}\label{namnare1}
|f(t)|^2 = g(0) = \int_{\rr d} \wh g(\xi) d\xi = \int_{\rr d} W_f(t,\xi) d \xi,
\end{equation}
in the last step using the definition \eqref{wd1}.

Concerning the numerator in \eqref{ifwigner1} we obtain using integration by parts and the fact \eqref{sobolevembedding1} that $f$ vanishes at infinity
\begin{equation}\label{intbypart1}
\begin{aligned}
\xi_j W_f(t,\xi) & = \frac{-1}{2 \pi i} \int_{\rr d} f(t+\tau/2) \overline{f(t-\tau/2)} \partial_{\tau_j} (e^{-2 \pi i \tau \cdot \xi}) d \tau \\
& = \frac{1}{2 \pi i} \int_{\rr d} e^{-2 \pi i \tau \cdot \xi} \partial_{\tau_j} \left( f(t+\tau/2) \overline{f(t-\tau/2)} \right) d \tau
\end{aligned}
\end{equation}
for $1 \leq j \leq d$.
Let us study the function $g_j := \partial_j f(t+\cdot/2) \overline{f(t-\cdot/2)}$ and
$$
h_j(\tau) := \partial_{\tau_j} \left( f(t+\tau/2) \overline{f(t-\tau/2)} \right)
= \frac{1}{2} (g_j(\tau)-\overline{g_j(-\tau)})
$$
for $1 \leq j \leq d$ fixed.
We have $\xi_j \wh f(\xi) \in L^2(\rr d)$ and thus $\partial_j f \in L^2(\rr d)$.
Since $f \in L^2(\rr d)$ the Cauchy--Schwarz inequality gives $g_j \in L^1(\rr d)$.
From above we know that $\mathscr F (\overline{f(t- \cdot/2)}) \in L^1(\rr d)$ and likewise we have
$$
\mathscr F(\partial_j f(t+ \cdot/2))(\xi) = 2^d e^{4 \pi i \xi \cdot t} \widehat{\partial_j f}(2 \xi)
= 2^d 2 \pi i e^{4 \pi i \xi \cdot t} 2 \xi_j \wh f(2 \xi) \in L^1(\rr d)
$$
because of \eqref{sobolev2}. Thus
$$
\wh g_j = \mathscr F (\partial_j f(t + \cdot/2)) * \mathscr F (\overline{f(t- \cdot/2)}) \in L^1(\rr d).
$$
Invoking \eqref{sobolevembedding1} we have proved that $g_j \in ( C_0 \cap L^1 \cap \mathscr F L^1)(\rr d)$ which means that Fourier's inversion formula holds for $g_j$ and thus $h_j$. Integration of \eqref{intbypart1} gives, denoting $f(t)=u(t)+i v(t)$,
\begin{equation}\label{taljare1}
\begin{aligned}
\int_{\rr d} \xi_j W_f(t,\xi) d \xi & = \frac1{2 \pi i} \int_{\rr d} \mathscr F h_j(\xi) d \xi = \frac{h_j(0)}{2 \pi i}  \\
& = \frac{g_j(0) - \overline{g_j(0)}}{4 \pi i} \\
& = \frac1{4 \pi i} \left( \partial_j f(t) \overline{f(t)} - \overline{\partial_j f(t)} f(t) \right) \\
& = \frac1{2 \pi} \left( u(t) \partial_j v(t) - \partial_j u(t) v(t) \right).
\end{aligned}
\end{equation}
Finally \eqref{ifwigner1} follows from a combination of \eqref{fasderivata1}, \eqref{namnare1} and \eqref{taljare1}.
\end{proof}

Next we will prove a version of Proposition \ref{ifwignerprop1} for functions $f \in \mathscr F \mathscr E'(\rr d)$. Note that this assumption is neither a generalization nor a special case of the assumptions of Proposition \ref{ifwignerprop1}.

\begin{prop}\label{ifwignerprop2}
If $f \in \mathscr F \mathscr E'(\rr d)$ then
\begin{equation}\label{ifwdmom2}
\frac1{2 \pi} \nabla \arg f (t) = \frac{\langle W_f(t,\cdot),\xi \rangle}{\langle W_f(t,\cdot),1 \rangle}, \quad f(t) \neq 0.
\end{equation}
\end{prop}
\begin{proof}
Let $t \in \rr d$ be fixed such that $f(t) \neq 0$, and let $f(t)=u(t) + i v(t)$. Set
$$
g_t:= f(t+\cdot/2) \overline{f(t-\cdot/2)}.
$$
By Corollary \ref{restrictioncor1} we have $\mathscr F g_t = W_f(t,\cdot)$ in $\mathscr D'(\rr d)$. The assumption $f \in \mathscr F \mathscr E'(\rr d)$ together with the Paley--Wiener--Schwartz theorem implies that $f$ extends to an entire function on $\cc d$ and there exist $N,A,C>0$ such that
\begin{equation}\nonumber
|f(x+iy)| \leq C \eabs{x}^N e^{A |y|}, \quad x,y \in \rr d.
\end{equation}
Hence the function $g_t$ extends to an entire function on $\cc d$ with bound
\begin{equation}\nonumber
\begin{aligned}
|g_t (z)| & = |f(t+z/2) \overline{f(t-z/2)}| \\
& \leq C \eabs{t+{\re z}/2}^N \eabs{t- {\re z}/2}^N e^{2A |{\im z}/2|} \\
& \leq C \eabs{t}^{2N} \eabs{{\re z}}^{2N} e^{A |{\im z}|}.
\end{aligned}
\end{equation}
By the Paley--Wiener--Schwartz theorem $\mathscr F g_t \in \mathscr E'(\rr d)$ and $\mathscr F g_t$ is supported in a fixed compact set independent of $t \in \rr d$. Since $\mathscr F g_t = W_f(t,\cdot)$ in $\mathscr D'(\rr d)$ we have $W_f(t,\cdot) \in \mathscr E'(\rr d)$.

Because $W_f(t,\cdot)$ has compact support, $\langle \mathscr F g_t,1 \rangle$ and $\langle \mathscr F g_t,\xi \rangle$ are well defined, where $\langle \mathscr F g_t,\xi \rangle$ means the vector $(\langle \mathscr F g_t,\xi_j \rangle )_{j=1}^d$ and $\xi_j: \rr d \mapsto \ro$ is the $j$th coordinate function. The inverse Fourier transform of a distribution $u \in \mathscr E'(\rr d)$ is the entire function $x \mapsto \langle u,e^{2 \pi i x \cdot } \rangle$ \cite[Theorem 7.1.14]{Hormander1} so we have
\begin{equation}\nonumber
g_t (\tau) = \mathscr F^{-1} \mathscr F g_t (\tau) = \langle \mathscr F g_t,e^{2 \pi i \tau \cdot} \rangle
\end{equation}
and in particular
\begin{equation}\label{denominator1}
\langle W_f(t,\cdot),1 \rangle = \langle \mathscr F g_t,1 \rangle = g_t (0) = |f(t)|^2 = u^2(t)+v^2(t).
\end{equation}
In a similar way we have for $1 \leq j \leq d$
\begin{equation}\label{nominator1}
\begin{aligned}
\langle W_f(t,\cdot),\xi_j \rangle
& = \langle \xi_j \mathscr F g_t,1 \rangle \\
& = \frac1{2 \pi i} \langle \mathscr F (\partial_j g_t),1 \rangle \\
& = \frac{\partial_j g_t (0)}{2 \pi i} \\
& = \frac1{4 \pi i} \left( \partial_j f(t) \overline{f(t)} - f(t) \overline{\partial_j f(t)} \right) \\
& = \frac{1}{2\pi}\left( u(t) \partial_j v(t) - \partial_j u(t) v(t) \right).
\end{aligned}
\end{equation}
Now \eqref{ifwdmom2} follows from \eqref{denominator1}, \eqref{nominator1} and \eqref{fasderivata1}.
\end{proof}

Finally we give a generalization of Proposition \ref{ifwignerprop2}.

\begin{prop}\label{ifwignerprop3}
Let
$$
f(t) = \exp(\pi i t \cdot A t) h(t)
$$
where $A \in \rr {d \times d}$ is symmetric, and $h \in \mathscr F \mathscr E'(\rr d)$. Then
\begin{equation}\label{ifwdmom3}
\frac1{2 \pi} \nabla \arg f (t) = \frac{\langle W_f(t,\cdot),\xi \rangle}{\langle W_f(t,\cdot),1 \rangle}, \quad f(t) \neq 0.
\end{equation}
\end{prop}
\begin{proof}
Let $t \in \rr d$ be fixed and satisfy $f(t) \neq 0$. We have
$$
\arg f(t) = \arg h(t) + \pi t \cdot A t,
$$
\begin{equation}\label{iftranslation1}
\nabla \arg f(t) = \nabla \arg h(t) + 2 \pi A t,
\end{equation}
and
$$
W_f(t,\xi) = W_h(t,\xi-A t).
$$
This yields
\begin{equation}\label{denominator2}
\langle W_f(t,\cdot),1 \rangle = \langle W_h (t,\cdot),1 \rangle
\end{equation}
and
\begin{equation}\label{nominator2}
\begin{aligned}
\langle W_f(t,\cdot),\xi_j \rangle & = \langle W_h (t,\cdot-A t), \xi_j \rangle \\
& = \langle W_h (t,\cdot), \xi_j \rangle + (A t)_j \langle W_h (t,\cdot), 1 \rangle.
\end{aligned}
\end{equation}
Finally \eqref{denominator2}, \eqref{nominator2}, Proposition \ref{ifwignerprop2} and \eqref{iftranslation1} give
\begin{equation}\nonumber
\begin{aligned}
\frac{\langle W_f(t,\cdot),\xi_j \rangle}{\langle W_f(t,\cdot),1 \rangle}
& = \frac{\langle W_h (t,\cdot), \xi_j \rangle + (A t)_j \langle W_h (t,\cdot), 1 \rangle}{\langle W_h(t,\cdot),1 \rangle} \\
& = \frac{\langle W_h (t,\cdot), \xi_j \rangle}{\langle W_h(t,\cdot),1 \rangle} + (A t)_j \\
& = \left( \frac1{2\pi} \nabla \arg h(t) + A t \right)_j \\
& = \left( \frac1{2\pi} \nabla \arg f(t) \right)_j.
\end{aligned}
\end{equation}
\end{proof}

\section*{Ackowledgement}

We wish to thank Professors J. Johnsen, O. Liess, L. Rodino and J. Toft for discussions and valuable advice that have helped improve the paper.


\end{document}